\documentclass[12pt,a4paper]{amsart}
\usepackage{graphicx,amssymb}
\usepackage{amsmath,amssymb,amscd}
\input xy
\xyoption{all}
\usepackage[all]{xy}
\usepackage{tikz}

\usepackage{hyperref}

\newcommand{\hra}{\hookrightarrow}
\newcommand{\ra}{\rightarrow}

\newcommand{\CC}{\mathbb C}

\newcommand{\cO}{\mathcal{O}}

\newcommand{\Ext}{\operatorname{Ext}}

\newcommand{\Hom}{\mbox{Hom}}

\newcommand{\Cliff}{\mbox{Cliff}}
\newcommand{\gr}{\mbox{gr}}
\newcommand{\rk}{\mbox{rk}}

\newcommand{\Cl}{\operatorname{Cliff}}
\theoremstyle{plain}
\newtheorem{theorem}{Theorem}[section]
\newtheorem{lem}[theorem]{Lemma}
\newtheorem{prop}[theorem]{Proposition}
\newtheorem{cor}[theorem]{Corollary}

\newtheorem{rem}[theorem]{Remark}

\numberwithin{equation}{section}
\begin{document}
\title[BN-theory for genus 5]{Higher rank BN-theory for curves of genus 5}

\author{H. Lange}
\author{P. E. Newstead}

\address{H. Lange\\Department Mathematik\\
              Universit\"at Erlangen-N\"urnberg\\
              Cauerstra\ss e 11\\
              D-$91058$ Erlangen\\
              Germany}
              \email{lange@mi.uni-erlangen.de}
\address{P.E. Newstead\\Department of Mathematical Sciences\\
              University of Liverpool\\
              Peach Street, Liverpool L69 7ZL, UK}
\email{newstead@liv.ac.uk}

\thanks{Both authors are members of the research group VBAC (Vector Bundles on Algebraic Curves). The second author 
would like to thank the Department Mathematik der Universit\"at 
         Erlangen-N\"urnberg for its hospitality}

\begin{abstract}
In this paper, we consider higher rank Brill-Noether theory for smooth curves of genus 5, obtaining new upper bounds for non-emptiness of Brill-Noether loci and many new examples.
\end{abstract}
\maketitle

\section{Introduction}
Let $C$ be a smooth complex projective curve and let $B(n,d,k)$ denote the {\it Brill-Noether locus} of stable bundles on $C$ of rank 
$n$ and degree $d$ with at least $k$ independent sections (for the formal definition, see Section \ref{back}). This locus has a natural 
structure as a subscheme of the moduli space of stable bundles on $C$ of rank $n$ and degree $d$.

In the case $n=1$, the Brill-Noether loci are classical objects. For $n>1$, 
the study began towards the end of the 1980s and the situation is much less clear, even on a general 
curve, and there is a great deal that is not known. The problem is completely 
solved only for $g\le3$ (see \cite{bgn,m1,m2} and Proposition \ref{pln2}), although there are strong results for hyperelliptic and bielliptic curves (see \cite{bmno} and \cite{b}) and for $g=4$ (see \cite{ln2}). 

Our object in this paper is to extend the results of \cite{ln2} to non-hyperelliptic curves of genus $5$. The main results of the paper concern new upper 
bounds on $k$ for the non-emptiness of $B(n,d,k)$ and the corresponding loci $\widetilde{B}(n,d,k)$ for 
semistable bundles. Since a complete answer is known for $d\le2n$ (see Proposition \ref{pln2}), it is sufficient in view of Serre duality to restrict to the range $2n<d\le4n$. To state our results, it is necessary to distinguish the case of 
trigonal curves from that of curves of Clifford index $2$. For trigonal curves, we have

\noindent{\bf Theorem \ref{thm4.7}.} \emph{Let $C$ be a trigonal curve of genus $5$. If $2n < d \leq 4n$ and $\widetilde B(n,d,k) \neq \emptyset$, then one of the following holds.
\begin{enumerate}
 \item[(i)] $2n < d \leq \frac{7n}{3}$ and $k \leq n +\frac{1}{4}(d-n)$;  
 \item[(ii)] $\frac{7n}{3} < d \leq \frac{5n}{2}$ and $k \leq d-n$;
 \item[(iii)] $\frac{5n}{2} \leq d \leq \frac{8n}{3}$ and $k \leq \frac{3n}{2}$;
 \item[(iv)] $\frac{8n}{3} \leq d < 3n$ and $ k \leq \frac{n}{2} + \frac{3d}{8}$;
 \item[(v)] $d = 3n$ and $k \leq2n$;
 \item[(vi)] $3n < d < 4n$ and $k \leq \frac{2}{5}(2n+d)$;
 \item[(vii)] $d = 4n$ and $k \leq 2n$.
\end{enumerate}
If $B(n,d,k)\ne\emptyset$, \rm{(v)} can be replaced by
\begin{enumerate}
\item[(v)$'$] $d=3n$ and either $k\le\frac{8n}5$ or $(n,d,k)=(1,3,2)$.
\end{enumerate}} 

These new upper bounds look complicated and can be best appreciated from  Figures 1 and 2 in Section \ref{bn}. They are probably not best possible, but they represent a substantial improvement on the known bound $k\le\frac12(d+n)$ (see Proposition \ref{p2.2}), especially in the range $2n\le d\le3n$. Note that there is no reason why the optimal upper bound should take a simple form.

For curves of Clifford index $2$, we  have a somewhat simpler result.

\noindent{\bf Theorem \ref{thm5.2}.} \emph{Let $C$ be a curve of genus $5$ and Clifford index $2$. If $2n < d \leq 4n$ and $\widetilde B(n,d,k) \neq \emptyset$, then one of the following holds.
\begin{enumerate}
 \item[(i)] $2n < d \leq \frac{7n}{3}$ and $ k \leq n + \frac{1}{4}(d-n)$;
 \item[(ii)] $\frac{7n}{3} < d \leq \frac{5n}{2}$ and $k < d-n$;
 \item[(iii)] $\frac{5n}{2} < d \leq 4n$ and $ k \leq n + \frac{1}{3}(d-n)$.
\end{enumerate}}

For a general curve, this theorem can be slightly improved in the range $\frac{5n}2<d<3n$ by replacing (iii) by parts (iii) and (iv) from Theorem \ref{thm4.7} (see Theorem \ref{thm5.3}). In any case, Theorem \ref{thm5.2} provides an improvement on the known bound $ k \leq n + \frac{1}{3}(d-n)$ (see Proposition \ref{p1.6}) in the range $2n<d<\frac{5n}2$. For a graphical representation, see Figures 3 and 4 in Section \ref{bn}. Again the results are almost certainly not best possible. 

We also produce 
a large number of examples of stable bundles which come close to attaining the upper bounds of Theorems \ref{thm4.7} and \ref{thm5.2}.
Many of these are constructed using elementary transformations, the only problem here being to prove stability.
Some of these were already established in \cite{ln2}, but others are new.
\bigskip

In Section \ref{back}, we give some background and describe some known results. In Section \ref{nonh}, 
we obtain upper bounds and also some existence results for non-hyperelliptic curves of genus $5$ in general.
Section \ref{trig} contains results for trigonal curves of genus $5$, which are especially strong in the range $2n<d\le3n$. For curves of genus $5$ and Clifford index $2$ (see Section \ref{cliff}), 
the results are quite similar for $2n<d\le3n$, but are considerably stronger for $3n<d\le4n$. In Section \ref{ext}, we consider bundles which maximise the number of sections for given rank and degree, 
bundles of ranks $2$ and $3$ and bundles of rank $n$ with $h^0 = n+1$. Finally, in Section \ref{bn}, we provide a graphical representation of our results.

Our methods are inspired in particular by those of \cite{bmno} and work of Mercat \cite{m1,m2}. 

We thank the referees for some helpful suggestions which have improved the presentation.

\section{Background and some known results}\label{back}

Let $C$ be a smooth projective curve of genus $g$. Denote by $M(n,d)$ the moduli space of stable vector 
bundles of rank $n$ and degree $d$ and by $\widetilde M(n,d)$ the moduli space of 
S-equivalence classes of semistable bundles of rank $n$ and degree $d$. For any integer
$k \geq 1$ we define 
$$
B(n,d,k) := \{ E \in M(n,d) \; | \; h^0(E) \geq k \}
$$
and
$$
\widetilde B(n,d,k) := \{ [E] \in \widetilde M(n,d)  \;|\; h^0( \gr E) \geq k \},
$$
where $[E]$ denotes the S-equivalence class of $E$ and $\gr E$ is the graded object defined 
by a Jordan-H\"older filtration of $E$. The locus $B(n,d,k)$ has an {\it expected dimension}
$$
\beta(n,d,k):=n^2(g-1)+1-k(k-d+n(g-1)),
$$
known as the {\it Brill-Noether number}.
For any vector bundle $E$ on $C$, we write $n_E$ for the rank of $E$, $d_E$ for the degree of $E$ and $\mu(E)=\frac{d_E}{n_E}$ for the slope of $E$. 
The vector bundle $E$ is said to be {\it generated} if the evaluation map $H^0(E)\otimes {\mathcal O_C}\to E$ is surjective. 

We recall the {\it dual span construction} (see, for example, \cite{bu} and \cite{m1}),
defined as follows. Let $L$ be a generated line bundle on $C$ with $h^0(L) \geq 2$.  Consider
the evaluation sequence
\begin{equation} \label{eq2.1}
0 \ra E_L^* \ra H^0(L) \otimes \cO_C \ra L \ra 0.
\end{equation}
Then $E_L$ is a bundle of rank $h^0(L) -1$ and degree $d_L$ with $h^0(E) \geq h^0(L)$.
 It is called the {\it dual span of} $L$ and is also denoted by $D(L)$. Although $E_L$ is not 
necessarily stable, this is frequently the case.

We begin by recalling some known results. In investigating the non-emptiness of $B(n,d,k)$ 
and $\widetilde B(n,d,k)$, it is sufficient by Serre duality and Riemann-Roch to consider the case $d\le n(g-1)$. 
For $g=0$ and $g=1$, there is nothing to be done. For $g=2$ and $g=3$, a complete solution is known 
(see \cite{bgn, m1, m2}). For $g=4$, some strong results were obtained in \cite{ln2}. For future reference, we note some facts here. 

\begin{prop}\label{pln2} Let $C$ be a curve of genus $g\ge3$ and suppose $k\ge1$.
\begin{enumerate}
\item[(i)] If $0<d<2n$, then $\widetilde B(n,d,k)\ne\emptyset$ if and only if $k-n\le\frac1g(d-n)$. Moreover
$B(n,d,k)\ne\emptyset$ under the same conditions except when $(n,d,k)=(n,n,n)$ with $n\ge2$.
\item[(ii)] If $C$ is non-hyperelliptic and $d=2n$, then $\widetilde B(n,d,k)\ne\emptyset$ if and only if $k\le\frac{ng}{g-1}$.
\item[(iii)] If $C$ is non-hyperelliptic and $d=2n$, then $B(n,d,k)\ne\emptyset$ if and only if $k\le\frac{n(g+1)}g$ or $(n,d,k)=(g-1,2g-2,g)$. Moreover $B(g-1,2g-2,g)=\{D(K_C)\}$.
\end{enumerate}
\end{prop}

This is contained in \cite{bgn,m1,m2} and is also included in \cite[Propositions 2.1 and 2.2]{ln2}.

\begin{cor} \label{c2.2}
If $2n < d \le 3n$ and $k-n \leq \frac{1}{g}(d - 2n)$, then $B(n,d,k) \neq \emptyset$.
\end{cor}

\begin{proof}
$B(n,d',k) \neq \emptyset$ for $n < d' \le 2n$ and $k-n \le \frac{1}{g}(d'-n)$ by Proposition \ref{pln2}(i) and (iii). Tensoring by an effective line bundle of degree 1 gives the result.
\end{proof}

For hyperelliptic curves, a complete solution for $0\le d<4n$ is contained in \cite[Propositions 2.1, 2.2 and 2.3]{ln2}. This fully covers the cases $g\le4$ and can be completed for $g=5$ by the following proposition.

\begin{prop} \label{p1.4}
Let $C$ be a hyperelliptic curve of genus $g \geq 5$ and $d = 4n$. Then
$B(n,d,k) \neq \emptyset$ if and only if either $k \leq 2n$ or $(n,d,k) = (1,4,3).$
Moreover, $\widetilde B(n,d,k) \neq \emptyset$ if and only if $k \leq 3n$.
\end{prop}

\begin{proof}
The necessity of the condition for $B(n,d,k)$ is a special case of \cite[Theorem 6.2(2)]{bmno}. 
The semistable case is easily deducible from this. 
For sufficiency, take $s = 2$ in \cite[Theorem 6.1]{bmno}. 
\end{proof}

We turn now to non-hyperelliptic curves.

\begin{prop} \label{p2.2}
Let $E$ be a semistable bundle on a non-hyperelliptic curve $C$ of rank $n$ and degree $d$.
\begin{enumerate}
 \item[(i)] If $1 \leq \mu(E) \leq 2g-3$, then $h^0(E) \leq \frac{1}{2} (d+n)$.
 \item[(ii)] If $\mu(E) \geq 3$, then $h^0(E) \leq d-n$.
\end{enumerate}
\end{prop}

\begin{proof}
See \cite[Proposition 3.1 and Lemma 3.2]{ln2}. (Part (i) is contained in \cite{re}.) 
\end{proof}

\begin{lem} \label{l2.3}
Suppose that $N$ is a generated line bundle on $C$ with $h^0(N) = 2$. Then, for any bundle $E$,
$$
h^0(N \otimes E) \geq 2h^0(E) - h^0(N^* \otimes E).
$$
In particular, if $E$ is either semistable with $\mu(E) < d_N$ or stable of rank $>1$ with $\mu(E) \leq d_N$, then
$$
h^0(N \otimes E) \geq 2h^0(E).
$$
\end{lem}

\begin{proof}
 We have an exact sequence
 $$
 0 \ra N^* \otimes E \ra H^0(N) \otimes E \ra N \otimes E \ra 0.
 $$
 The first assertion follows immediately from this. The second assertion follows, if we note that under the stated conditions $h^0(N^* \otimes E) = 0.$
\end{proof}

\begin{prop} \label{prop2.4}
Let $C$ be a trigonal curve of genus $g$ and $3n < d < 5n$.  
If $k \leq 2\left\lfloor n + \frac1{g}(d -4n)\right\rfloor$ and $(n,d,k)\ne(n,4n,2n)$ or $(n,4n,2n-1)$, then $B(n,d,k) \neq \emptyset$.
\end{prop}

\begin{proof}
We know by Proposition \ref{pln2} that $B(n,d',k') \neq \emptyset$ if $0 < d'< 2n$ and $k' \leq n + \frac{1}{g}(d'-n)$, except when $(n,d',k') = (n,n,n)$ with $n \geq 2$.
Now take $N$ in Lemma \ref{l2.3} to be a trigonal bundle. The result follows from this and the fact that $B(n,d,k+1) \subset B(n,d,k)$.
\end{proof}

For curves of higher Clifford index we have a stronger version of Proposition \ref{p2.2}(i).

\begin{prop} \label{p1.6}
Suppose that $\Cl(C) \geq 2$ and $E$ is a semistable bundle on $C$ of rank $n$ and slope $\mu = \frac{d}{n}$. 
\begin{enumerate}
\item[(i)]
If $2 + \frac{2}{g-4} \leq \mu \leq 2g -4 - \frac{2}{g-4}$, then
$$
h^0(E) \leq \frac{d}{2}.
$$
\item[(ii)] 
If $1 \leq \mu \leq 2 + \frac{2}{g-4}$, then 
$$
h^0(E) \leq \frac{1}{g-2}(d -n) + n.
$$
\end{enumerate}
\end{prop}

For the proof, see \cite[Theorem 2.1]{m}. We have stated this result in full, although only (ii) is relevant for $g=5$.

\begin{prop} \label{p1.7}
Let $C$ be a bielliptic curve and $n,d$ and $k$ positive integers.
\begin{enumerate}
\item[(i)] If $k \leq \frac{d}{2}$, then there exists a semistable bundle $E$ of rank $n$ and degree $d$ with $h^0(E) \geq k$.
\item[(ii)] If $k < \frac{d}{2}$, then there exists a stable bundle of rank $n$ and degree $d$ with 
$h^0(E) \geq k$.
\end{enumerate}
\end{prop}

This is \cite[Theorem 3.1]{m} and is due to Ballico \cite[Theorem 5.3 and Proposition 5.4]{b}.

A common method of construction is that of elementary transformations. We have in particular
\begin{prop} \label{p2.5}
 Let $C$ be a curve of genus $g\ge2$ and $L_1, \dots, L_n$ line bundles of degree $d$ on $C$ with $L_i \not \simeq L_j$ for $i \neq j$ and let $t > 0$. Then
 \begin{enumerate}
  \item[(i)] there exist stable bundles $E$ fitting into an exact sequence
  $$
  0 \ra L_1 \oplus \cdots \oplus L_n \ra E \ra \tau \ra 0
  $$
  where $\tau$ is a torsion sheaf of length $t$;
  \item[(ii)] there exist stable bundles $E$ fitting into an exact sequence
  $$
  0 \ra E \ra L_1 \oplus \cdots \oplus L_n \ra \tau \ra 0
  $$
  with $\tau$ as above.
 \end{enumerate}
\end{prop}

\begin{proof}
(i) is a particular case of \cite[Th\'eor\`eme A.5]{m3}. (ii) can be deduced by replacing each $L_i$ by $K_C \otimes L_i$ and using Serre duality.
(For a general curve, this is proved in \cite{te}.)
\end{proof}

Finally, we have the following simple lemma.

\begin{lem}\label{lbb}
If $\widetilde B(n,d,k)\ne\emptyset$, then $B(n',d',k')\ne\emptyset$ for some $(n',d',k')$ with $n' \leq n,\, \frac{d'}{n'}=\frac{d}{n}$ and $\frac{k'}{n'}\ge\frac{k}{n}$.
\end{lem}
\begin{proof} Let $E$ be a semistable bundle of type $(n,d,k)$. At least one factor in any Jordan-H\"older filtration of $E$ must belong to some $B(n',d',k')$ as specified in the statement.
\end{proof}

\section{Non-hyperelliptic curves of genus $5$}\label{nonh}

In view of the facts cited in section \ref{back}, we need to consider only the case $2 < \mu \leq 4$.

\begin{lem} \label{l5.1}
Let $C$ be a non-hyperelliptic curve of genus $5$ and $L = K_C(-p)$ for some $p \in C$.
Then $L$ is generated, $E_L$ is stable of rank $3$ and degree $7$ and
$$
h^0(E_L) = 4.
$$
\end{lem}

\begin{proof}
$L$ is generated since $C$ is non-hyperelliptic; moreover $h^0(L) = 4$ by Riemann-Roch. Hence $E_L$ has rank $3$ and degree $7$; moreover $E_L$ is stable by \cite[Lemma 3.7]{ln2}.
The fact that $h^0(E_L) \geq 4$ follows from dualizing \eqref{eq2.1}. If $\Cliff(C) = 2$, then $h^0(E_L) \leq 4$ by Proposition \ref{p1.6} (ii). Suppose therefore that 
$C$ is trigonal with trigonal bundle $T$. Since $h^0(T) = 2$, Serre duality and Riemann-Roch give $h^0(K_C \otimes T^*) = 3$. Hence $h^0(L \otimes T^*) \geq 2$ 
and there exist non-zero homomorphisms $T \ra L$. Thus we obtain a non-zero homomorphism $D(L) \ra D(T)$, i.e. $E_L \ra T$.
Since $E_L$ is stable, this must be surjective and we have an exact sequence
$$
0 \ra F \ra E_L \ra T \ra 0.
$$
The rank-2 bundle $F$ is semistable, since a line subbundle of $F$ of degree $\geq 3$ would contradict the stability of $E_L$.
But now $h^0(F) \leq \frac{5}{2}$ by Proposition \ref{pln2}(ii), which implies that $h^0(E_L)\le4$.
\end{proof}

The following is the case $g=5$ of \cite[Lemma 3.7(2)]{ln2}.

\begin{lem}\label{lel}
Let $C$ be a non-hyperelliptic curve of genus $5$ and $L=K_C(-p)$ for some $p\in C$. Suppose that $E$ is a bundle of rank $n$ and degree $d$ with $h^1(E\otimes L)=0$ and $h^0(E)>n+\frac14(d-n)$. Then $h^0(E_L^*\otimes E)>0$.
\end{lem}

The following proposition incorporates the case $g=5$ of \cite[Lemma 3.8]{ln2}.

\begin{prop} \label{p3.2}
Let $C$ be a non-hyperelliptic curve of genus $5$ and $L = K_C(-p)$ for some $p \in C$.
Let $E$ be a semistable bundle of rank $n$ and degree $d$ with slope $\mu > 2$.  
Suppose that
$$
h^0(E) > n + \frac{1}{4}(d-n).
$$
Then
\begin{enumerate}
 \item[(i)] $h^0(E_L^* \otimes E) > 0$;
 \item[(ii)] $\mu > \frac{7}{3}$;
 \item[(iii)] if $\mu < \frac{5}{2},\; E_L$ can be embedded as a subbundle in $E$.
\end{enumerate}
\end{prop}

\begin{proof}
Since $E\otimes L$ is semistable of slope $>2g-1$, we have $h^1(E\otimes L)=0$. The assertion (i) now follows from Lemma \ref{lel}. The inequality $\mu\ge\frac73$ and (iii) follow from Lemma \ref{l5.1}.
When $\mu = \frac{7}{3}$, (iii) implies that $E_L$ can be embedded as a subbundle of 
$E$. Hence $E/E_L$ satisfies the hypotheses of the proposition and so by induction every factor of the Jordan-H\"older filtration of $E$ is isomorphic to $E_L$. Since 
$h^0(E_L) = 4$ by Lemma \ref{l5.1}, this contradicts the hypothesis.
\end{proof}

\begin{prop} \label{p3.4}
Let $C$ be a non-hyperelliptic curve of genus $5$. 
Suppose that $k = n + \frac{1}{4}(d-n)$. Suppose further that $2 < \frac{d}{n} \leq \frac{7}{3}$. 
If $B(n,d,k) \neq \emptyset$, then $(n,d,k) = (3,7,4)$. Moreover,
\begin{equation} \label{eq3.1}
B(3,7,4) = \{ E_L \;|\; L = K_C(-p) \; \mbox{for some} \; p \in C \}. 
\end{equation}
\end{prop}

\begin{proof} (This follows the same lines as \cite[Proposition 6.1]{ln2}, but is more complicated, so we give the proof in full.)
Suppose $E \in B(n,d,k)$. Note that we have $h^0(E) = k$ by Proposition \ref{p3.2}.
We first claim that $E$ is generated.

If not, there exists an exact sequence 
$$
0 \ra F \ra E \ra \CC_q \ra 0
$$
with $h^0(F) = k$. Let $L = K_C(-p)$. Since $E \otimes L$ is stable with slope $> 9$, it follows that $E \otimes L$ is generated. Hence
$$
h^1(F \otimes L) = h^1(E \otimes L) = 0.
$$
It now follows from Lemma \ref{lel} that $h^0(E_L^* \otimes F) > 0$. Hence $E \simeq E_L$. This contradicts the assumption that $E$ is not generated.

It follows that  we have an exact sequence
$$
0 \ra G^* \ra H^0(E) \otimes \cO_C \ra E \ra 0
$$
with $n_G = k-n,\; d_G = d$ and $h^0(G) \geq k$. It follows that $K_C \otimes G^*$ has rank $k-n$ and
\begin{eqnarray*}
 h^0(K_C \otimes G^*) &=& h^0(G) -d + 4(k-n)\\
 & \geq & k-d + d-n = n_G.\\
\end{eqnarray*}
Any such bundle necessarily has a section with a zero. So $K_C \otimes G^*$ admits a line subbundle $M$ with $h^0(M) \geq 1$ and $d_M \geq 1$ and we get the diagram
$$
\xymatrix{
0 \ar[r] & E^* \ar[d]_{\alpha} \ar[r] & W \otimes \cO_C \ar[d] \ar[r] & G \ar[r] \ar[d] & 0\\
0  \ar[r] & H^* \ar[r]  & V \otimes \cO_C  \ar[r] \ar[d] & K_C \otimes M^* \ar[r] \ar[d] & 0\\
&  & 0 & 0 &
}
$$
where $W$ is a subspace of $H^0(G)$ of dimension $k$ and $V$ is the image of $W$ in $H^0(K_C \otimes M^*)$.
Now $K_C \otimes M^*$ is not isomorphic to $\cO_C$, since $h^0(G^*) = 0$. Hence $\dim V \geq 2$ and $d_{K_C \otimes M^*} \geq 3$, since $C$ is non-hyperelliptic.
Since $d_{K_C \otimes M^*} \leq 7$, we have also $\dim V \leq 4$ with equality only if $K_C \otimes M^* \simeq K_C(-p)$ for some $p \in C$. 

If $\alpha = 0$, then $E^*$ maps into $W' \otimes \cO_C$, where $W = W' \oplus V'$ and $V'$ maps isomorphically to $V$. It follows that $V' \otimes \cO_C$ 
maps to a trivial direct summand of $G$ contradicting the fact that $h^0(G^*) =0$. So $\alpha \neq 0$.

If $\dim V = 2$, then $\alpha(E^*)$ is a quotient line bundle of $E^*$ of degree $\leq -3$, contradicting the stability of $E$. 

If $\dim V = 4$, we can write $K_C \otimes M^* \simeq K_C(-p) =: L$ and then $H \simeq E_L$, which is stable with $\mu(E_L) \geq \mu(E)$. Hence $E \simeq E_L$.

It remains to consider the case $\dim V = 3$. If $K_C \otimes M^* \simeq K_C(-p) =: L$ for some $p \in C$, then $H^*$ is a subbundle of $E^*_L$, so 
there exists a non-zero homomorphism $E^* \ra E_L^*$ implying $E \simeq E_L$. Finally suppose that $\dim V = 3$ and $V = H^0(K_C \otimes M^*)$. Then $d_H = d_{K_C \otimes M^*} = 5$ or 6. If $\alpha$ has rank 2, this contradicts the stability of $E$. 
So suppose $\rk \, \alpha = 1$.
Since $H$ is generated and $H^0(H^*) =0$, every quotient line bundle of $H$ has degree $\geq 3$, again contradicting the stability of $E$.
 \end{proof}

\begin{prop} \label{p5.6}
Let $C$ be a non-hyperelliptic curve of genus $5$ and $E$ a semistable bundle of rank $n$ and degree $d$ with slope $\mu(E) > \frac{7}{3}$. Then
$$
h^0(E) \leq d-n.
$$ 
\end{prop}

\begin{proof}
The proof is by induction on $n$. Note that by Proposition \ref{p2.2}(ii) we can assume that $\mu(E)<3$. 

For  $n = 1$, the result is trivial. 
For $n = 2$, the only possibility is $d=5$ and then $h^0(E) \leq 3$ by Proposition \ref{p2.2}(i). For $n=3$, the only possibility is $d = 8$ 
and then $h^0(E) \leq 5$ by Proposition \ref{p2.2}(i).

Suppose now $n \geq 4$ and the proposition is proved for rank $ \leq n-1$.
Then there exists an exact sequence
$$
0 \ra F \ra E \ra G \ra 0
$$
in which $F$ is a proper subbundle of maximal slope and is stable.
Moreover $G$ is semistable.
To see this, suppose that $G'$ is a quotient bundle of $G$ with $\mu(G') < \mu(G)$. Set $F' := \ker(E \ra G')$. Note that $n_{F'} >  n_F$. 
We have 
$$
\mu(F') \leq \mu(F).
$$
Also
$$
\mu(F') \leq \mu(G')
$$
by semistability of $E$. So
\begin{eqnarray*}
d &=& n_F \mu(F) + (n-n_F) \mu(G)\\
&>& n_F \mu(F') + (n-n_F) \mu(G')\\
& \geq & n_{F'} \mu(F') + (n-n_{F'}) \mu(G') =d,
\end{eqnarray*}
a contradiction.

If $h^0(E)>d-n$, then, by Proposition \ref{p3.2}, there exists a non-zero homomorphism $E_L \ra E$. So
$$
\mu(F) \geq \mu(E_L) = \frac{7}{3},
$$
with equality only if $F \simeq E_L$. Also
$$
\mu(G) > \frac{7}{3}
$$
by semistability of $E$. So $G$ satisfies the inductive hypotheses and $F$ does, unless $F \simeq E_L$ in which case $h^0(F) = d_F -n_F$. 
This completes the inductive step and hence the proof.
\end{proof}

\begin{lem} \label{l5.3}
Let $C$ be a non-hyperelliptic curve of genus $5$.
There exists a stable bundle $U$ on $C$ of rank $2$ with $d_U = 5$ and $h^0(U) = 3$ if and
only if $C$ is trigonal. Moreover  $U \simeq D(K_C \otimes T^*)$, where $T$ is the unique trigonal bundle.
\end{lem}

\begin{proof}
Let $E$ be a stable bundle of rank 2 with $d_E = 5$ and $h^0(E) =3$. Using Proposition \ref{pln2}(ii), it is easy to see that the 
evaluation map $H^0(E) \otimes \cO_C \ra E$ is surjective. It follows that $E \simeq E_M$ for some line 
bundle $M$ of degree 5 with $h^0(M) = 3$. Then $K_C \otimes M^*$ has degree 3 and $h^0 = 2$.
In other words, it is the trigonal bundle.

Hence, if there exists a stable rank-2 bundle $E$ of degree 5 with $h^0(E) = 3$, then $C$ is trigonal and $E\simeq D(K_C \otimes T^*)$. 
Now there exist non-trivial extensions
\begin{equation} \label{e3.2}
0 \ra K_C \otimes T^{*2} \ra U \ra T \ra 0.
\end{equation}
For any such extension, $U$ is stable. Moreover, there exists an extension for which all sections of $T$ lift to $U$, since
\begin{equation} \label{e3.4}
H^0(T) \otimes H^0(T^2) \ra H^0(T^3) 
\end{equation}
is not surjective. To see this, note first that $H^0(T) \otimes H^0(T^2)$ has dimension 6 and $h^0(T^3) = 5$. The kernel of \eqref{e3.4} is isomorphic 
to $H^0(T)$ by the base-point free pencil trick, so has dimension 2.
\end{proof}

\begin{lem} \label{l5.5}
Let $C$ be a non-hyperelliptic curve of genus $5$.
Let $E$ be a semistable bundle on $C$ of slope $\mu,\, \; 2 < \mu < 3$ such that $h^0(E) > n + \frac{1}{4}(d-n)$.
Then one of the following occurs.
\begin{enumerate}
 \item[(i)] $\mu > \frac{7}{3}$ and $E_L$ can be embedded as a subbundle of $E$;
 \item[(ii)] $\mu \geq \frac{8}{3}$ and $E$ has a rank-$3$ subbundle $E'$ of degree $8$;
\item[(iii)] $C$ is trigonal, $\mu \geq \frac{5}{2}$ and $U$ can be embedded as a subbundle of $E$.
\end{enumerate}
\end{lem}

\begin{proof}
By Proposition \ref{p3.2}(i) and (ii), $h^0(E_L^* \otimes E) > 0$ and $\mu>\frac73$. Let $E_L \ra E$ be a non-zero homomorphism. If this homomorphism does not 
embed $E_L$ as a subbundle, then either (ii) holds or the image of $E_L$ in $E$ is stable of  rank 2 and degree $5$. By Lemma \ref{l5.3}, this leaves (iii) as the 
only possibility.
\end{proof}

\begin{prop} \label{p3.6}
Let $C$ be a non-hyperelliptic curve of genus $5$. Then $B(n,d,k) \neq \emptyset$ in the following 
cases.
\begin{enumerate}
\item[(i)] $(n,d,k) = (4r +s,8r+2s+1,5r+s)$ for $1 \leq r \leq 4, \; s \geq 0$;
\item[(ii)] $(n,d,k) = (4r+s,8r+2s+2,5r+s)$ for $1 \leq r \leq 4, \; s \geq 4r+1$;
\item[(iii)] $(n,d,k) = (n,2n+1,n+r)$ for $n \geq 5r,\; r \geq 1$.
\end{enumerate}
\end{prop}

\begin{proof}
These are special cases of \cite[Propositions 3.3 and 3.6 and Example 3.9]{ln2}.
\end{proof}

\begin{prop} \label{p3.7}
Let $C$ be a curve of genus $5$. Then $B(n,d,k) \neq \emptyset$ for $3n < d < 4n$ and $k \leq d -2n$. 
\end{prop}

\begin{proof}
Take $L_1, \dots, L_n$ pairwise non-isomorphic line bundles of degree 4 with $h^0(L_i) = 2$. Such bundles exist for all $n$ on any curve of genus 5. 
The result follows from Proposition \ref{p2.5}(ii).
\end{proof}

The argument used in the proof of Proposition \ref{p3.7} does not work for $d = 4n$. However, we have the following proposition.

\begin{prop} \label{p3.10}
Let $C$ be a curve of genus $5$. Then
\begin{enumerate}
 \item[(i)] $\widetilde B(n,4n,k) \neq \emptyset$ for $k \leq 2n$;
 \item[(ii)] $B(n,4n,k) \neq \emptyset$ for $k \leq \frac{6n}{5}$ and $B(4,16,5) \neq \emptyset$;
 \item[(iii)] $B(n,4n,k) \neq \emptyset$ for $k < 2n$ if $C$ is general.
\end{enumerate}
\end{prop}

\begin{proof}
For (i) we can take a direct sum of line bundles of degree 4 with $h^0=2$. (ii) follows from Proposition \ref{pln2}(iii) by tensoring with an effective line bundle of degree 2. 
(iii) is proved in \cite{te}.
\end{proof}

\section{Trigonal curves of genus $5$}\label{trig}

In this section, let $C$ be a trigonal curve of genus 5 with trigonal bundle $T$ and let $U = D(K_C \otimes T^*)$.

\begin{lem} \label{l4.1}
The bundle $U$ admits a unique line subbundle $M$ of degree $2$. Moreover $M \simeq K_C \otimes {T^*}^2$.
\end{lem}

\begin{proof}
By \eqref{e3.2}, it is sufficient to show that $h^0(U^* \otimes T)=1$. However, $U^* \otimes T$ is stable of rank 2 and degree 1, so $h^0(U^* \otimes T) \leq 1$ 
by Proposition \ref{pln2}(i). The result follows.
\end{proof}

\begin{lem} \label{l5.4}
Let $L= K_C(-p)$ for some $p \in C$. Then there exist surjective homomorphisms $E_L \ra T$ and $E_L \ra U$. Moreover, $U$ is the only quotient of $E_L$ 
of rank $2$ and degree $5$.
\end{lem}

\begin{proof}
The existence of $E_L\ra T$ was proved in the proof of Lemma \ref{l5.1}. Using \eqref{eq2.1}, we have an exact sequence
$$
0 \ra H^0(E_L^* \otimes U) \ra H^0(L) \otimes H^0(U) \ra H^0(L \otimes U).
$$
Now $h^0(L) = 4, h^0(U) = 3$ and $h^0(L \otimes U) = 11$ by Riemann-Roch. So there exists a non-zero homomorphism $E_L \ra U$. If this homomorphism has rank 1, then $E_L$ 
has a quotient of rank 1 and degree $\leq 2$, contradicting stability. If $E_L \ra U$ has rank 2, but is not surjective, then $E_L$ has a line subbundle of 
degree $\geq 3$, again contradicting stability.

Now any rank-2 and degree-5 quotient bundle of $E_L$ must be stable with $h^0 \geq 3$. By Lemma \ref{l5.3} the only such bundle is $U$.
\end{proof}

\begin{prop} \label{p4.3}
Let $E$ be a semistable bundle with $\frac{5}{2} \leq \mu(E) < 3$. Then 
$$
h^0(E) \leq \max \left\{ \frac{n}{2} + \frac{3d}{8}, \frac{3n}{2} \right\}.
$$
\end{prop}

\begin{proof}
Suppose first that $\mu(E) \geq \frac{8}{3}$.
In this case $2 < \mu(K_C \otimes T^* \otimes E^*) \leq \frac{7}{3}$. Hence by Proposition \ref{p3.2},
$$
h^0(K_C \otimes T^* \otimes E^*) \leq n + \frac{1}{4}(4n - d) = 2n - \frac{d}{4}.
$$
By Lemma \ref{l2.3},
$$
h^0(T \otimes E) \geq 2 h^0(E).
$$
Hence by Riemann-Roch,
$$
d-n = \chi(T \otimes E) \geq 2 h^0(E) - \left(2n - \frac{d}{4} \right),
$$ 
which implies the result in this case. 

If $\frac{5}{2} \leq \mu(E)< \frac{8}{3}$, we argue in the same way, but now use Proposition 
\ref{p5.6} to show that $h^0(K_C \otimes T^* \otimes E^*) \leq 4n-d$. Then
$$
d-n = \chi(T \otimes E) \geq 2h^0(E) - 4n + d.
$$
Hence $h^0(E) \leq \frac{3n}{2}$.
\end{proof}

\begin{lem} \label{l4.4}
 Let $L = K_C(-p)$ for some $p \in C$. Then the multiplication map
\begin{equation} \label{eq4.1}
H^0(T) \otimes H^0(K_C \otimes E_L^*) \ra H^0(T \otimes K_C \otimes E_L^*)                                                                    
\end{equation}
is surjective with kernel $H^0(T^* \otimes K_C \otimes E_L^*)$ of dimension $4$.
\end{lem}

\begin{proof}
From the sequence
$$
0 \ra T^* \ra H^0(T) \otimes \cO_C \ra T \ra 0,
$$
we see that the kernel of \eqref{eq4.1} is $H^0(T^* \otimes K_C \otimes E_L^*)$. Since $T^* \otimes K_C \otimes E_L^*$ is a stable bundle of 
rank 3 and degree 8, Proposition \ref{p4.3} implies that $h^0(T^* \otimes K_C \otimes E_L^*) \leq 4$.  

Now $h^0(T) = 2,\; h^0(K_C \otimes E_L^*) = h^0(E_L) - d_{E_L} + 12 = 9$ and $h^0(T \otimes K_C \otimes E_L^*) = 14$ by Riemann-Roch.
The result follows.
\end{proof}

\begin{lem}   \label{lem4.5}
 Let $L = K_C(-p)$ for some $p \in C$. Then the multiplication map
 \begin{equation} \label{eq4.2}
H^0(U) \otimes H^0(K_C \otimes E_L^*) \ra H^0(U \otimes K_C \otimes E_L^*)
 \end{equation}
is surjective with kernel $h^0(T \otimes E_L^*)$ of dimension $2$.
\end{lem}
\begin{proof}
From the sequence
$$
0 \ra (K_C \otimes T^*)^* \ra H^0(U) \otimes \cO_C \ra U \ra 0,
$$
we see that the kernel of \eqref{eq4.2} is $H^0(T \otimes E_L^*)$. Since $T \otimes E_L^*$ is stable of rank 3 and degree 2, we have $h^0(T \otimes E_L^*) \leq 2$
by Proposition \ref{pln2}(i). 

Now $h^0(U) = 3,\; h^0(K_C \otimes E_L^*) = 9$ and, by Riemann-Roch, $h^0(U \otimes K_C \otimes E_L^*) = 25$. The result follows.
\end{proof}

\begin{prop}  \label{prop4.6}
Suppose that $E$ is a semistable bundle of rank $2r$ and degree $5r$ with $h^0(E) = 3r$. Then
$$
E \simeq \bigoplus_{i=1}^r U.
$$
\end{prop}

\begin{proof}
First we claim that $h^0(U^* \otimes E) > 0$.

Since $U \simeq D(K_C \otimes T^*)$ by Lemma \ref{l5.3}, we have an exact sequence
$$
0 \ra U^* \otimes E \ra H^0(K_C \otimes T^*) \otimes E \ra K_C \otimes T^* \otimes E \ra 0.
$$
Now $h^0(K_C \otimes T^*) = 3$ and $h^0(E) = 3r$. On the other hand, $T \otimes E^*$ is a semistable bundle of rank $2r$ and degree $r$, so by Proposition \ref{pln2}(i), $h^0(T \otimes E^*) \leq \frac{9r}{5}$. It follows by Riemann-Roch that $h^0(K_C \otimes T^* \otimes E) \leq 9r - \frac{r}{5}$. 
Hence $h^0(U^* \otimes E) \geq \frac{r}{5}$, which proves the claim.

The proof of the proposition is by induction on $r$. For $r = 1$, we have a non-zero homomorphism $U \ra E$ which is necessarily an isomorphism. 

Suppose therefore that $r \geq 2$ and the result is proved for the case $r-1$. A non-zero homomorphism $U \ra E$ is necessarily an injection onto a subbundle.
By the inductive hypothesis, $E/U$ is isomorphic to $r-1$ copies of $U$ and we have an exact sequence
$$
0 \ra U \ra E \ra \bigoplus_{i=1}^{r-1} U \ra 0;
$$
moreover all sections of $\bigoplus_{i-1}^{r-1}U$ must lift. If the extension is non-trivial, it follows that the map
$$
H^0(U) \otimes H^0(K_C \otimes U^*) \ra H^0(K_C \otimes U \otimes U^*)
$$
is not surjective. Its kernel is $H^0(K_C \otimes U^* \otimes (K_C \otimes T^*)^*) = H^0(T \otimes U^*)$, which has dimension $\leq 1$ by Proposition \ref{pln2}(i). 

Now $h^0(U) = 3,\; h^0(K_C \otimes U^*) = 6$ and $h^0(K_C \otimes U \otimes U^*) = 17$ by Riemann-Roch, since $h^0(U\otimes U^*)=1$. This gives a contradiction.
It follows that the extension is trivial and $E \simeq \bigoplus_{i-1}^r U$.
\end{proof}

\begin{lem} \label{l5.8}
Let $E$ be a stable bundle with $\mu(E) = 3$ and $n \geq 2$. Then
$$
 h^0(E) \leq \frac{8n}{5}.
$$
\end{lem}

\begin{proof}
Suppose $E$ is a stable bundle of rank $n \geq 2, \mu(E) = 3$ and $h^0(E) > \frac{8n}{5}$. By Lemma \ref{l2.3},
$$
h^0(T \otimes E) \geq 2h^0(E) > \frac{16n}{5}.
$$
By Riemann-Roch and Serre duality, 
$$
h^0(K_C \otimes T^* \otimes E^*) > \frac{6n}{5}.
$$
Since $K_C \otimes T^* \otimes E^*$ is stable of slope 2, this contradicts 
Proposition \ref{pln2}(iii), unless $E \simeq K_C \otimes T^* \otimes D(K_C)^*$.
However $h^0(D(K_C)) = 5$. It follows that in this case $h^0(T \otimes E) = 13$. 
Hence $h^0(E) \leq 6 < \frac{8n}{5}$, since $n = 4$, and the result still holds.
\end{proof}

\begin{lem} \label{l4.5}
There exist generated line bundles of degree $4$ with $h^0 =2$. 
\end{lem}

\begin{proof}
Let $Q := K_C \otimes T^*(-p)$ for some $p \in C$. Certainly $h^0(Q) = 2$. If $Q$ is not generated, then
$Q = T(q)$ for some $q \in C$. So $K_C = T^2(p+q)$ which is true for a unique divisor $p+q$. This implies the assertion.
\end{proof}

\begin{prop} \label{p4.6}
Let $E$ be a semistable bundle of rank $n$ and degree $d$. Suppose $3 < \mu(E) \leq 4$. Then
$$
h^0(E) \leq \left\{ \begin{array}{lll} 
                     \frac{2}{5}(2n + d) & if & \mu(E) < 4\\
                     2n & if & \mu(E) = 4.
                     \end{array} \right.
$$
\end{prop}

\begin{proof}
Let $Q$ be a generated line bundle with $d_Q=4$ and $h^0=2$. Suppose first that $\mu(E) < 4$. Then, since $0 < \mu(K_C \otimes Q^* \otimes E^*) < 1$, we have
$$
h^0(K_C \otimes Q^* \otimes E^*) \leq n + \frac{1}{5}(3n - d)
$$ 
by Proposition \ref{pln2}(i). Therefore, by Riemann-Roch,
$$
h^0(Q \otimes E) \leq \frac{4}{5}(2n+d).
$$
From Lemma \ref{l2.3}, we get
$$
2 h^0(E) \leq h^0(Q \otimes E),
$$
which implies the assertion in this case.

Now suppose $\mu(E) = 4$. In view of Lemma \ref{lbb}, we can suppose moreover that $E$ is stable. If $n=1$, then obviously $h^0(E) \leq 2$. If $n \geq 2$, then $\mu(K_C \otimes Q^* \otimes E^*) = 0$; hence
$h^0(K_C \otimes Q^* \otimes E^*) = 0$. By Serre duality and Riemann-Roch, we obtain $h^0(Q\otimes E)=4n$, giving the assertion.
\end{proof}

The following theorem summarizes the results on upper bounds obtained above (see Figures 1 and 2 in Section \ref{bn}).

\begin{theorem} \label{thm4.7}
Let $C$ be a trigonal curve of genus $5$. If $2n < d \leq 4n$ and $\widetilde B(n,d,k) \neq \emptyset$, then one of the following holds.
\begin{enumerate}
 \item[(i)] $2n < d \leq \frac{7n}{3}$ and $k \leq n +\frac{1}{4}(d-n)$;  
 \item[(ii)] $\frac{7n}{3} < d \leq \frac{5n}{2}$ and $k \leq d-n$;
 \item[(iii)] $\frac{5n}{2} \leq d \leq \frac{8n}{3}$ and $k \leq \frac{3n}{2}$;
 \item[(iv)] $\frac{8n}{3} \leq d < 3n$ and $ k \leq \frac{n}{2} + \frac{3d}{8}$;
 \item[(v)] $d = 3n$ and $k \leq2n$;
 \item[(vi)] $3n < d < 4n$ and $k \leq \frac{2}{5}(2n+d)$;
 \item[(vii)] $d = 4n$ and $k \leq 2n$.
\end{enumerate}
If $B(n,d,k)\ne\emptyset$, \rm{(v)} can be replaced by
\begin{enumerate}
\item[(v)$'$] $d=3n$ and either $k\le\frac{8n}5$ or $(n,d,k)=(1,3,2)$.
\end{enumerate}
\end{theorem}

\begin{proof}
(i) follows from Proposition \ref{p3.2}, (ii) is Proposition \ref{p5.6}, for (iii) and (iv) see Proposition \ref{p4.3},
for (v) and (v)$'$ combine Lemma \ref{l5.8} with the existence of $T$, and for (vi) and (vii) see Proposition \ref{p4.6}.
\end{proof}

\begin{rem} \label{r4.8}
 {\rm
 Let $\Cliff_n(C)$ be the rank-$n$ Clifford index as defined for example in \cite{ln}. It follows from Theorem \ref{thm4.7}
 that all bundles computing $\Cliff_n(C)$ must have degree $3n$ and $h^0= 2n$. In fact, by \cite[Corollary 4.8]{ln}, there is only 
 one such bundle, namely $\oplus_{i=1}^n T$.
 }
\end{rem}

\begin{prop}  \label{prop4.12}
$B(2,7,4) = \emptyset$.
\end{prop}

\begin{proof}
Let $E \in B(2,7,4)$. We prove first that 
$$h^0(T^* \otimes E) > 0.
$$
In fact, $H^0(T^* \otimes E)$ is the kernel of the multiplication map
$$
H^0(T) \otimes H^0(E) \ra H^0(T \otimes E).
$$
Now $h^0(T) = 2, \; h^0(E) = 4$ and, by Riemann-Roch,
$$
h^0(T \otimes E) = h^0(K_C \otimes T^* \otimes E^*) + 7 + 6 - 8.
$$
Since $K_C \otimes T^* \otimes E^*$ is a stable bundle of rank 2 and degree $3$, we have
$$
h^0(K_C \otimes T^* \otimes E^*) \leq 2
$$ 
by Proposition \ref{pln2}(i) and hence $h^0(T \otimes E) \leq 7$. This proves the assertion.

It follows that we have an exact sequence
\begin{equation} \label{eq4.3}
0 \ra T \ra E \ra M \ra 0
\end{equation}
with $d_M = 4$, $h^0(M) = 2$ and all sections of $M$ lift.
Hence the map
$$
H^0(M) \otimes H^0(K_C \otimes T^*) \ra H^0(K_C \otimes M \otimes  T^*)
$$
is not surjective. However, $h^0(K_C \otimes T^*) = 3$ and the kernel is 
$H^0(K_C \otimes T^* \otimes M^*)$, which has dimension $\leq 1$. 

Moreover, $h^0(K_C \otimes M \otimes T^*) = 5$ by Riemann-Roch, a contradicction. This proves that
$B(2,7,4) = \emptyset$.
\end{proof}

\begin{prop} \label{prop4.13}
$B(2,8,4) \neq \emptyset$. 
\end{prop}

\begin{proof}
We consider non-trivial extensions \eqref{eq4.3} with $M = K_C \otimes T^*$. 
Then certainly $E$ is semistable. Hence $h^0(E) \leq 4$ by Proposition \ref{p4.6}.
Since $h^0(M) = 3$, it follows that not all sections of $M$ lift. So the canonical map
\begin{equation} \label{eq4.4}
H^1(M^* \otimes T) \ra \Hom(H^0(M), H^1(T)) 
\end{equation}
is injective. Noting that $M$ is generated, choose a 2-dimensional subspace $V$ of $H^0(M)$ which generates $M$. We obtain
an exact sequence
$$
0 \ra H^0(\cO_C) \ra V \otimes H^0(M) \ra H^0(M^2).
$$
Since $\dim V = 2, \; h^0(M) = 3$ and $h^0(M^2) = 6$, it follows that 
$V \otimes H^0(M)  \ra H^0(M^2)$ is not surjective and has cokernel of dimension 1. 
Equivalently, the dual map
$$
H^1(M^* \otimes T) \ra \Hom (V,H^1(T))
$$
has kernel of dimension 1. Taking \eqref{eq4.3} to be the extension corresponding to a non-zero 
element $\xi$ of this kernel, we obtain a unique bundle $E$ with $h^0(E) = 4$. 

We need to show that $E$ is stable.
If $E$ is not stable, then $E$ must have a tetragonal subbundle $Q$ admitting a non-zero homomorphism $Q\to M$. This implies that 
\eqref{eq4.3} becomes trivial, when pulled back by $Q \ra M$. It follows that the element of 
$H^1(M^* \otimes T)$ defining \eqref{eq4.3} is in the kernel of the map
$$
H^1(M^* \otimes T) \ra H^1(Q^* \otimes T) \ra \Hom(H^0(Q), H^1(T)).
$$
$H^0(Q)$ and $V$ are both subspaces of codimension 1 of $H^0(M)$ and $\xi$ goes to 
zero  under the restriction of \eqref{eq4.4} to both $H^0(Q)$ and $V$. Hence $H^0(Q) = V$, in which 
case $V$ does not generate $M$. This contradicts the assumption. The conclusion is that
$B(2,8,4) \neq \emptyset$.
\end{proof}

\begin{prop} \label{p4.4}
$$
B(n,3n,k) \neq \emptyset \;\;\; \mbox{for} \;\; k \leq 2\left\lfloor\frac{6n}{5}\right\rfloor-n.
$$
Moreover, $B(4,12,6) \neq \emptyset.$
\end{prop}

\begin{proof}
We know, by Proposition \ref{pln2}(iii), that $B(n,2n,k') \neq \emptyset$ for $k' \leq \frac{6n}{5}$. If $E \in B(n,2n,k')$, then 
$h^0(T \otimes E) \geq 2h^0(E)$ by Lemma \ref{l2.3} and so
$$
h^0(K_C \otimes T^* \otimes E^*) \geq 2h^0(E) -n.
$$
This proves the first statement. For the second, note that $K_C \otimes T^* \otimes D(K_C)^*$ 
belongs to
$B(4,12,6)$.
\end{proof}

\begin{prop} \label{p4.9}
$B(n,d,k) \neq \emptyset$ in the following cases.
\begin{enumerate}
 \item[(i)] $(n,d,k) = (4r+s,12r+3s-1,6r+s-1)$ for $1 \leq r \leq 4, \; s \geq 0$;
 \item[(ii)] $(n,d,k) = (4r+s,12r+3s-2,6r+s-2)$ for $1 \leq r \leq 4,\; s \geq 4r+1$;
 \item[(iii)] $(n,d,k) = (n,3n-1,n +2r-1)$ for $n \geq 5r,\; r\geq 1$.
\end{enumerate}
\end{prop}

\begin{proof}
By Proposition \ref{p3.6}(i), we have with the hypotheses of (i), $B(4r+s,8r+2s+1,5r+s) \neq \emptyset$. Using Lemma \ref{l2.3} with $N = T$, it follows that 
$B(4r+s,20r+5s+1,10r+2s) \neq \emptyset$. The result follows by Serre duality and Riemann-Roch. (ii) and (iii) follow similarly from Proposition \ref{p3.6}(ii) and (iii).
\end{proof}

\begin{prop} \label{p4.11}
 For any $p \in C$, there exist exact sequences
 \begin{equation} \label{e4.1}
0 \ra U \ra E \ra \CC_p \ra 0  
 \end{equation}
 with $E$ stable. Hence $B(2,6,3) \neq \emptyset$.
\end{prop}

\begin{proof}
Consider exact sequences \eqref{e4.1}. If $E$ is not stable, then it possesses a line subbundle $N$ of degree 3. By Lemma \ref{l4.1}, it follows that we have a diagram
$$
\xymatrix{
0 \ar[r] & U  \ar[r] & E  \ar[r] & \CC_p \ar[r] \ar@{=}[d] & 0\\
0  \ar[r] & K_C \otimes {T^*}^2 \ar[r] \ar@{^{(}->}[u] & N  \ar[r] \ar@{^{(}->}[u] & \CC_p \ar[r]  & 0\\
}
$$
and the embedding of $K_C \otimes {T^*}^2$ in $U$ is unique up to a scalar. It follows that such a diagram cannot exist for the general extension \eqref{e4.1}.
\end{proof}

\begin{prop} \label{p4.12}
$B(2r,6r-1,3r-1) \neq \emptyset$ for any $r \geq 1$.\ 
\end{prop}

\begin{proof}
Choose $r$ pairwise non-isomorphic bundles $E_1, \dots, E_r \in B(2,6,3)$. These exist by Proposition \ref{p4.11}. Let $E$ be an elementary transformation
$$
0 \ra E \ra E_1 \oplus \cdots \oplus E_r \ra \CC_p \ra 0
$$
for some $p \in C$ such that the homomorphisms $E_i\to\CC$ are all non-zero. Since the partial direct sums of the $E_i$ are the only subbundles of  $E_1 \oplus \cdots \oplus E_r$ of slope $3$, it follows that every subbundle of $E$ has slope $<3$. Hence $E \in B(2r,6r-1,3r-1)$.
\end{proof}

\begin{prop} \label{p4.13}
For any $p \in C$, there exist exact sequences 
\begin{equation} \label{e4.2}
0 \ra U \oplus U \ra E \ra \CC_p \ra 0 
\end{equation}
with $E$ stable. Hence $B(4,11,6) \neq \emptyset$.
\end{prop}

\begin{proof}
Consider exact sequences \eqref{e4.2}. If $E$ is not stable, there exists a diagram 
$$
\xymatrix{
0 \ar[r] & U \oplus U  \ar[r] & E  \ar[r] & \CC_p \ar[r] \ar@{=}[d] & 0\\
0  \ar[r] & F' \ar[r] \ar@{^{(}->}[u] & F  \ar[r] \ar@{^{(}->}[u] & \CC_p \ar[r]  & 0\\
}
$$
with $n_F \leq 3$ and $\mu(F) = 3$.

If $n_F = 1$, then $d_{F'} = 2$ and $h^0(F'^* \otimes (U \oplus U)) = 2$ by Lemma \ref{l4.1}.
If $n_F  = 2$, then $d_{F'} = 5$. In this case, $F' \simeq U$ and again $h^0(F'^* \otimes (U \oplus U)) =2$.
In both cases the diagram cannot exist for a general extension \eqref{e4.2}.
If $n_F = 3$, then $d_{F'} =8$. This contradicts the semistability of $U \oplus U$.
\end{proof}

\begin{prop} \label{p4.14}  
For any $p,q \in C$, there exist exact sequences
\begin{equation} \label{e4.3}
0 \ra E_L \ra E \ra \CC_q \ra 0 
\end{equation}
with $E$ stable, where $L = K_C(-p)$. In particular $B(3,8,4) \neq \emptyset$.
\end{prop}

\begin{proof}
Suppose that $F$ is a proper subbundle of $E$ with $\mu(F) \geq \frac{8}{3}$. Then we have a diagram 
$$
\xymatrix{
0 \ar[r] & E_L  \ar[r] & E  \ar[r] & \CC_q \ar[r] \ar@{=}[d] & 0\\
0  \ar[r] & F' \ar[r] \ar@{^{(}->}[u] & F  \ar[r] \ar@{^{(}->}[u] & \CC_q \ar[r]  & 0\\
}
$$
If $n_F = 2$, we must have $d_{F'} = 5$. This contradicts the stability of $E_L$.

If $n_{F} = 1$, we must have $d_{F'} =2$. It follows from Lemma \ref{l5.4} that $F'\simeq T(-p)$. Moreover, $h^0(E_L \otimes F'^*) \leq 2$ by Proposition \ref{pln2}(i), since 
$E_L \otimes F'^*$ is stable of slope $\frac{1}{3}$. It follows that the diagram cannot exist for a general extension \eqref{e4.3}.
\end{proof}

\section{Curves of Clifford index $2$}\label{cliff}

Suppose that $C$ is a curve of genus 5 and Clifford index 2. The main difference from the trigonal case is that the bundles $T$ and $U$ do not exist. 
However $C$ is tetragonal and possesses a one-dimensional family of line bundles of degree 4 with $h^0 = 2$. In this case we have the following 
slight improvement of Proposition \ref{p5.6}.

\begin{prop} \label{p5.1}
Let $C$ be a curve of genus $5$ and Clifford index $2$ and $E$ a semistable bundle of rank $n$ and degree $d$ with $\mu(E) > \frac{7}{3}$. Then 
$$
h^0(E) < d-n.
$$
\end{prop}

\begin{proof}
The proof of Proposition \ref{p5.6} goes through with the improved inequality, noting that, by Proposition \ref{p1.6}(ii) and Lemma \ref{l5.3} there are no bundles $E$ of ranks 1, 2 or 3 with  $\mu(E)> \frac{7}{3}$ and $h^0(E) = d-n$. 
\end{proof}

\begin{theorem} \label{thm5.2}
Let $C$ be a curve of genus $5$ and Clifford index $2$. If $2n < d \leq 4n$ and $\widetilde B(n,d,k) \neq \emptyset$, then one of the following holds.
\begin{enumerate}
 \item[(i)] $2n < d \leq \frac{7n}{3}$ and $ k \leq n + \frac{1}{4}(d-n)$;
 \item[(ii)] $\frac{7n}{3} < d \leq \frac{5n}{2}$ and $k < d-n$;
 \item[(iii)] $\frac{5n}{2} < d \leq 4n$ and $ k \leq n + \frac{1}{3}(d-n)$.
\end{enumerate}
\end{theorem}

\begin{proof}
 (i) follows from Proposition \ref{p3.2}, for (ii) see Proposition \ref{p5.1} and (iii) is Proposition \ref{p1.6}(ii) for $g=5$.
\end{proof}

\begin{rem} \label{r5.3}
 {\rm
 It follows from Theorem \ref{thm5.2} that all bundles computing $\Cliff_n(C)$ in this case have degree $4n$ and $h^0 = 2n$.
 For rank 2, this was proved in \cite[Proposition 5.7]{ln3}. At least on a general curve of genus 5, there exist     
stable bundles of rank 2 and degree 8 with $h^0= 4$ by \cite[Section 3]{bf}. This answers a question 
raised in \cite[Remark 5.8]{ln3}).
 }
\end{rem}

\begin{theorem} \label{thm5.3}
Let $C$ be a general curve of genus $5$. If $2n < d \leq 4n$ and $\widetilde B(n,d,k) \neq \emptyset$, then one of the following holds.
\begin{enumerate}
\item[(i)] $2n < d \leq \frac{7n}{3}$ and $ k \leq n + \frac{1}{4}(d-n)$;
 \item[(ii)] $\frac{7n}{3} < d \leq \frac{5n}{2}$ and $k < d-n$;
 \item[(iii)] $\frac{5n}{2} < d \leq \frac{8n}{3}$ and $k \leq \frac{3n}{2}$;
 \item[(iv)] $\frac{8n}{3} < d < 3n$ and $ k \leq \frac{n}{2} + \frac{3d}{8}$;
 \item[(v)] $3n \leq d \leq 4n$ and $k \leq  n+ \frac{1}{3}(d-n)$.
\end{enumerate}
\end{theorem}

\begin{proof}
(i), (ii) and (v) follow from Theorem \ref{thm5.2}. (iii) and (iv) follow from Theorem \ref{thm4.7}(iii) and (iv) by semicontinuity. 
\end{proof}

\begin{prop} \label{p5.4}
Let $C$ be a curve of genus $5$ and Clifford index $2$. Then 
\begin{equation} \label{e5.1}
B(2,6,3) \neq \emptyset.
\end{equation}
\end{prop}

\begin{proof}
Let $M$ be a line bundle of degree 6 with $h^0(M) = 3$. Then $M$ is generated and the bundle $E_M$, defined as in \eqref{eq2.1}, has rank 2 and degree 6. If $E_M$ is not stable, then 
we have a diagram
$$
\xymatrix{
0 \ar[r] & E_M^*  \ar[r] \ar[d]& H^0(M) \otimes \cO_C \ar[d] \ar[r] & M \ar[r] \ar@{=}[d] & 0\\
0  \ar[r] & N \ar[r] \ar[d] & F  \ar[r] \ar[d] & M \ar[r]  & 0\\
& 0 & 0 &&
}
$$
with $N$ a line bundle with $d_N \leq -3$. It follows that $F$ is a generated rank-2 bundle with $d_F \leq 3$ and $h^0(F) \geq 3$. 
This cannot exist on a non-trigonal curve, so $E_M \in B(2,6,3)$.
\end{proof}

\begin{cor} \label{c5.5}
$B(2r,6r-1,3r-1) \neq \emptyset$ for any $r \geq 1$. 
\end{cor}

\begin{proof}
The proof is exactly the same as for Proposition \ref{p4.12}, using Proposition \ref{p5.4} and the fact that $\beta(2,6,3)>0$. 
\end{proof}

\begin{cor} \label{c5.6}
$B(2r,6r+1,3r) \neq \emptyset$ for any $r \geq 1$. 
\end{cor}

\begin{proof}
We consider extensions
$$
0 \ra E_1 \oplus \cdots \oplus E_r \ra E \ra \CC_p \ra 0,
$$
where $E_1, \dots, E_r \in B(2,6,3)$ and are pairwise non-isomorphic. The general extension gives a stable bundle.
\end{proof}

\begin{lem} \label{l5.6}
Let $C$ be a curve of genus $5$ and Clifford index $2$. Let $L_i = K_C(-p_i)$ for $1 \leq i \leq r$, where $p_1, \dots, p_r$ are distinct points of $C$. 
Then every proper subbundle $F$ of $E_{L_1} \oplus \cdots \oplus E_{L_r}$, which is not isomorphic to a partial direct sum of factors of $E_{L_1} \oplus \cdots \oplus E_{L_r}$, has 
$$
d_F \leq \frac{7}{3} n_F - 1.
$$
\end{lem}

\begin{proof}
The proof is by induction on $r$. If $r=1$, then $n_F = 1$ or 2.

If $n_F = 1$, then $d_F \leq 1$, since otherwise $E_{L_1}/F$ is a quotient of $E_{L_1}$ of rank 2 and degree $ 5$ with $h^0 \geq 3$. 
It is easy to see that this quotient must be stable, which contradicts Lemma \ref{l5.3}.
If $n_F = 2$, then $d_F \leq 3$, since otherwise $E_{L_1}$ would have a quotient bundle of rank 1 and degree $\leq 3$ with $h^0 \geq 2$. This contradicts 
$\Cliff(C) = 2$.

Now suppose $r \geq 2$ and the lemma is proved for $r-1$ factors. Consider the projection $\pi: F \ra E_{L_1}$. We can assume without loss of generality that $\pi \neq 0$.
If $\rk \;\pi = 3$, then by induction
$$
d_F \leq 7 + \frac{7}{3}(n_F -3) -1 = \frac{7}{3}n_F -1.
$$
If $\rk \;\pi = 1$, then
$$
d_F \leq 1 + \frac{7}{3}(n_F -1) < \frac{7}{3} n_F - 1.
$$
If $\rk \;\pi = 2$, then
$$
d_F \leq 3 + \frac{7}{3}(n_F -2) < \frac{7}{3}n_F - 1.
$$
This completes the proof.
\end{proof}

\begin{prop} \label{p5.7}
Let $C$ be a curve of genus $5$ and Clifford index $2$. Suppose $r \geq 1$, $p \in C$ and $L_1, \dots, L_r$ are as in Lemma  \ref{l5.6}. Let
$$
0 \ra E_{L_1} \oplus \cdots \oplus E_{L_r} \ra E \ra \CC_p \ra 0
$$
be the extension classified by  $(e_1,\ldots,e_r)$, where the $e_i\in\Ext(\CC_p,E_{L_i})$ are all non-zero. Then $E$ is stable. Hence $$B(3r,7r+1, 4r) \neq \emptyset.$$
\end{prop}

\begin{proof}
It follows from Lemma \ref{l5.6} that any proper subbundle $F$ of $E$ has $d_F \leq \frac{7}{3}n_F$. Hence $E$ is stable.  
\end{proof}

\begin{prop} \label{p5.10}
Let $C$ be a curve of genus $5$ and Clifford index $2$. Then $B(3,9,4) \neq \emptyset$ and $B(3,12,4) \neq \emptyset$. 
\end{prop}

\begin{proof}
Let $L = K_C(-p)$ for some $p \in C$. Consider an exact sequence
$$
0 \ra E \ra E_L(q) \ra \CC_q \ra 0
$$
for any $q \in C$. Since $E_L$ as a subsheaf of $E_L(q)$ consists of local sections vanishing at $q$, we have a sheaf inclusion $E_L \hra E$. This implies that $h^0(E) \geq 4$. 

Now by Lemma \ref{l5.6}, any subbundle of $E_L(q)$ of rank 1 has degree $\leq 2$ and any subbundle of rank 2 has degree $\leq 5$. These do not contradict 
the stability of $E$. So $E \in B(3,9,4)$. Moreover, $E(r) \in B(3,12,4)$ for any $r \in C$.
\end{proof}

\begin{rem}
 {\rm 
The stable bundles constructed in this section exist on any curve of genus 5 and Clifford index 2, in particular on a bielliptic curve of genus 5.
Many of them have $k > \frac{d}{2}$, so lie outside the scope of Ballico's result (Proposition \ref{p1.7}). This means that the bundles and their sections are not lifted
from the corresponding elliptic curve.
}
\end{rem}

\section{Extremal bundles, bundles of low rank and $k=n+1$}\label{ext}

Let $C$ be a non-hyperelliptic curve of genus 5. The bundles $D(K_C), E_L$ and $Q$ are extremal in the sense that they take the maximal value of $\frac{h^0}{n}$
for bundles of the same slope (see the figures in Section \ref{bn}). By Proposition \ref{pln2}(iii), $D(K_C)$ is the only stable bundle representing the point $(2, \frac{5}{4})$
in the BN-map. By Proposition \ref{p3.4}, the only bundles on the line in the BN-map joining $(2,\frac{5}{4})$ to $(\frac{7}{3},\frac{4}{3})$ are $D(K_C)$ and the bundles $E_L$.

When $C$ is trigonal, then $U$ and $T$ are also extremal. By Remark \ref{r4.8}, $T$ is the only stable bundle representing the point $(3,2)$.

\begin{prop} \label{p6.1}
Let $C$ be a non-hyperelliptic curve of genus $5$ and $E$ a semistable bundle of rank $n$ and degree $d$ with slope $\mu(E) > \frac{7}{3}$. Then 
$$
h^0(E) < d-n,
$$
unless $C$ is trigonal and $E \simeq \oplus_i U$ or $E \simeq \oplus_i T$.
\end{prop}

\begin{proof}
For $\Cl(C)=2$, this is Proposition \ref{p5.1}. For $C$ trigonal, the proof follows that of Proposition \ref{p5.6}. For $n= 1,2,3$ the result is clear. Continuing with the proof, we see that either $h^0(E) < d-n$ or there is an exact sequence
$$
0 \ra F \ra E \ra G \ra 0
$$
with $F \simeq E_L, U$ or $T$ and $G \simeq \oplus_i T$ or $\simeq \oplus_i U$ and all sections of $G$ must lift.

If $F \simeq T$ and $G \simeq \oplus_i T$, then $E \simeq \oplus_i T$ by Remark \ref{r4.8}. If $F \simeq U$ and $G \simeq \oplus_i T$, then 
$\frac{5}{2} < \mu(E) < 3$ and the result follows from Proposition \ref{p4.3}.
If $F \simeq U$ and $G \simeq \oplus_i U$, then $E \simeq \oplus_i U$ by Proposition \ref{prop4.6}.

If $F \simeq E_L$ and $G \simeq \oplus_i T$, then by Lemma \ref{l4.4} not all sections of $G$ lift.
Finally, if $F \simeq E_L$ and $G \simeq \oplus_i U$, then by Lemma \ref{lem4.5} not all sections of $G$ lift. This completes the inductive step and hence the proof.
\end{proof}

\begin{prop} \label{p6.2}
Suppose $0 < d < 16$. Then $B(2,d,k) \neq \emptyset$ if and only if 
$$
\beta(2,d,k) := 17 - k(k-d+8) \geq 0
$$
with the following exceptions.
\begin{enumerate}
\item[(i)]  $(d,k) = (2,2)$ or $(14,8)$, in which case $\widetilde B(2,d,k) \neq \emptyset$ but
$B(2,d,k) = \emptyset$;
\item[(ii)] $B(2,8,4) \neq \emptyset$ for $C$ general and for $C$ trigonal, but possibly not for all $C$ with $\Cl(C) =2$; $\widetilde B(2,8,4) \neq \emptyset$ for all $C$;
\item[(iii)] if $C$ is trigonal, $B(2,5,3) \neq \emptyset$ and $B(2,11,6) \neq \emptyset$.
\end{enumerate}
\end{prop}

\begin{proof}
Suppose first that $0 < d \leq 8$. For $k = 1$ and 2, the result follows from Proposition 
\ref{pln2}. For $k=3$, $\beta(2,d,k) \geq 0$ is equivalent to $d \geq 6$. The result follows from 
Lemma \ref{l5.3} and Propositions \ref{p4.11} and \ref{p5.4}. 
For $k=4$, $\beta(2,d,k) \geq 0$ is equivalent to $d \geq 8$ and the result follows from 
Propositions \ref{prop4.12} and \ref{prop4.13}, Theorem \ref{thm5.2}(iii) and Remark \ref{r5.3}.
For $k \geq 5$, $\beta(2,d,k) < 0$ for all $d \leq 8$ and all $\widetilde B(2,d,k) = \emptyset$ by Theorems \ref{thm4.7} and \ref{thm5.2}.

For $d > 8$ we use Serre duality and Riemann-Roch.
\end{proof}

\begin{prop} \label{p6.3}
Suppose $0< d < 24$. Then $B(3,d,k) \neq \emptyset$ if and only if
$$
\beta(3,d,k) := 37 - k(k-d + 12) \geq 0
$$
with the following exceptions.
\begin{enumerate}
\item[(i)] $(d,k) = (3,3)$ or $(21,12)$, in which case $\widetilde B(3,d,k) \neq \emptyset$ but 
$B(3,d,k) = \emptyset$;
\item[(ii)] if $C$ is trigonal, $\widetilde B(3,9,4) \neq \emptyset$ and $\widetilde B(3,15,7) \neq \emptyset$, 
but it is possible that $B(3,9,4) = \emptyset$ and $B(3,15,7) = \emptyset$;
\item[(iii)] if $\Cl(C) = 2$, it is possible that $B(3,9,5)\neq \emptyset$ and $B(3,15,8) \neq \emptyset$;
\item[(iv)] $B(3,10,5)$ and $B(3,14,7)$  might be empty;
\item[(v)] $\widetilde B(3,12,5) \neq \emptyset$; moreover, $B(3,12,5) \neq \emptyset$ for $C$ general, but there might be curves for which $B(3,12,5) = \emptyset$;
\item[(vi)] $\widetilde B(3,12,6) \neq \emptyset$, but $B(3,12,6)$ might be empty;
\item[(vii)] if $C$ is trigonal, $B(3,10,6),B(3,11,6),B(3,13,7)$ and $B(3,14,8)$ might be non-empty.
\end{enumerate}
\end{prop}

\begin{proof}
Suppose first that $0 < d \leq 12$. For $k = 1,2,3$ the result follows from Proposition \ref{pln2}.

For $k = 4$, $\beta(3,d,k) \geq 0$ is equivalent to $d \geq 7$. We have $B(3,d,4) = \emptyset$ 
for $d \leq 6$ by Proposition \ref{pln2}, $B(3,7,4) \neq \emptyset$ by Lemma \ref{l5.1} and
$B(3,8,4) \neq \emptyset$ by Propositions \ref{p4.14} and \ref{p5.7}.
If $\Cl(C) = 2$, then $B(3,9,4) \neq \emptyset$ and $B(3,12,4) \neq \emptyset$ by Proposition \ref{p5.10}.
For $C$ trigonal, we see that 
$\widetilde B(3,9,4) \neq \emptyset$ by taking the direct sum of a bundle in $B(2,6,3)$ and a
line bundle  of degree 3 and $h^0 =1$. 
It is possible that $B(3,9,4) = \emptyset$, but $B(3,12,4) \neq \emptyset$ by Proposition \ref{prop2.4}. 
For $d \geq 10, \, d \neq 12$, $B(3,d,4) \neq \emptyset$ by Proposition \ref{p3.7} for any $C$. 

For $k = 5$,  $\beta(3,d,k) \geq 0$ is equivalent to $d \geq 10$.  $\widetilde B(3,d,5) = \emptyset$ 
for $d \leq 8$ by Theorem \ref{thm4.7}(i)-(iv) and Theorem \ref{thm5.2}. $B(3,9,5)= \emptyset$ 
if $C$ is trigonal by Theorem \ref{thm4.7}(v)$'$. We do not know whether this is true for any curve of Clifford index 2. 
$B(3,11,5) \neq \emptyset$ in all cases by Proposition \ref{p3.7}. For $d = 12$, use Proposition \ref{p3.10}.

For $k = 6$, $\beta(3,d,k) \geq 0$ is equivalent to $d \geq 12$. If $\Cl(C) = 2$, $\widetilde B(3,d,6) = \emptyset$ for $d \leq 11$ by Theorem \ref{thm5.2}
and for $d \leq 8$ if $C$ is trigonal by Theorem \ref{thm4.7}. $\widetilde B(3,12,6) \neq \emptyset$ in all cases by Proposition \ref{p3.10}. 
If $C$ is trigonal, then $B(3,9,6) = \emptyset$ by Theorem \ref{thm4.7}(v)$'$.

For $k \geq 7$, $\beta(3,d,k) < 0$ for all $d \leq 12$ and $\widetilde B(3,d,k) = \emptyset$ by Theorems \ref{thm4.7} and \ref{thm5.2}.

For $d > 12$ we use Serre duality and Riemann-Roch.
\end{proof}

It would be possible to extend this analysis to $k=4$, but the details would be complicated. However, there is one case where there is a simple answer.

\begin{prop}  \label{p6.4} Let $C$ be a non-hyperelliptic curve of genus $5$. Then
$\widetilde B(4,10,5) \neq \emptyset$. If $\Cl(C) = 2$, then $B(4,10,5) \neq \emptyset$ and $\widetilde B(4,10,5)=B(4,10,5)$. 
\end{prop}

\begin{proof}
Since $\beta(4,10,5) = 10 > 0$, $\widetilde B(4,10,5) \neq \emptyset$ for a general curve by \cite[Theorem 5.1]{bbn}. This holds for any curve by semicontinuity.
If $B(4,10,5) = \emptyset$, then by Lemma \ref{lbb}, $B(2,5,3) \neq \emptyset$. By Lemma \ref{l5.3}, this is not possible for $\Cl(C) = 2$. Indeed, in this case, there are no strictly semistable bundles of rank $4$ and degree $10$ with $h^0\ge5$.
\end{proof}

This result can be completed and partially extended to the case $k = n+1$ for all $n$.

\begin{prop} \label{p6.5}
Let $C$ be a non-hyperelliptic curve of genus $5$ and suppose $n \geq 2$.
\begin{enumerate}
 \item[(i)] If $\beta(n,d,n+1) < 0$, then $\widetilde B(n,d,n+1) = \emptyset$, except when $C$ is trigonal and $(n,d,n+1) = (2,5,3)$.
 \item[(ii)] If $\beta(n,d,n+1) \geq 0$, then 
 \begin{enumerate}
 \item[(a)] $\widetilde B(n,d,n+1) \neq \emptyset$,
 \item[(b)] $B(n,d,n+1) \neq \emptyset$, except possibly when 
 \begin{itemize}
  \item[] $n \geq 10,n\mbox{ even }, d = 2n+4$,
\item[] $n \geq 9,n\mbox{ divisible by }3, d = 2n+3$,
\item[] $n=8, d=18\mbox{ or }20$,
  \item[] $n=6, \,14 \leq d \leq 16$,
  \item[] $n=4,\,d=10,\; C  \; trigonal$, 
   \item[] $n=3,\,d=9$, $C$ trigonal.
   \end{itemize}
 \end{enumerate}
\end{enumerate}
\end{prop}

\begin{proof}
Suppose first $n \geq 5$. Then $\beta(n,d,n+1) < 0$ is equivalent to $d \leq n+4$. (i) now follows from Proposition \ref{pln2} and (ii)(a) follows from
\cite[Theorem 5.1]{bbn} for $C$ general and hence for any $C$. 

For (ii)(b), $B(n,d,n+1) \neq \emptyset$ for $n+5 \leq d \leq 2n$ by Proposition \ref{pln2}. Tensoring by an effective line bundle gives the same result for $2n+5 \leq d \leq 3n$
and for $d = 4n$. For $3n+1 \leq d < 4n$, see Proposition \ref{p3.7}. For $d = 2n+1$, see Proposition \ref{p3.6}(iii) and for $d = 2n+2$ and $n \geq 9$, Proposition \ref{p3.6}(ii). For $d>4n$, tensor by an effective line bundle. Note also that, if $(n,d)=1$, then (ii)(b) follows from (ii)(a).

For $n=4$, we have $\beta(4,d,5) = 5d -40$. If $d \leq 7$, then $\widetilde B(4,d,5) = \emptyset$ by Proposition \ref{pln2}(i). For $d = 8$, note that $D(K_C) \in B(4,8,5)$.
For $d = 9,10,11$, see Propositions \ref{p3.6}(i), \ref{p6.4}, \ref{p4.12} and Corollary \ref{c5.5}. For $d \geq 12$, tensor by an effective line bundle.
For $d = 14$, when $C$ is trigonal, we still have $B(4,14,5) \neq \emptyset$ by Proposition \ref{prop2.4}.

For $n=3$, see Proposition \ref{p6.3} and for $n = 2$, see Proposition \ref{p6.2}.
\end{proof}

\section{BN-map for genus $5$}\label{bn}

The following figures are the most significant part of the BN-map for non-hyperelliptic curves of genus 
5.
The map plots $\lambda = \frac{k}{n}$ against $\mu = \frac{d}{n}$.

We begin with the trigonal case (Figures 1 and 2).

\begin{tikzpicture}
\path[draw][->] (2,-0.06) -- (2,8.1) node[pos=0.99,above] {$\lambda$} node[pos=.34,left] {$4/3$} node[pos=0.99,left] {$2$};
\path[draw][->] (0,0 ) -- (10.2,0) node[pos=1.02,right] {$\mu$} node[pos=.6,below] {$5/2$} node[pos=0.21,below] {$2$} node[pos=0.98,below] {$3$} node[pos=0,left] {$1$}; 
               
\path[draw] (10, 0) -- (10,8) node[pos=1,right] {$T$} node[pos=0.63,right] {$13/8$} node[pos=0.59,right] {$8/5$} node[pos=0.2,right] {$6/5$} 
node[pos=0.5,right] {\eqref{e4.1}, $K_C \otimes T^* \otimes D(K_C)^*$}
node[pos=0.4,right] {$7/5$}; 
%\path[draw] (2, 0) -- (10,2);
\path[draw][dashed] (2,8) -- (10,8);
%\path[draw][dashed] (2,6) -- (10,6)  node[pos=0,left] {$7/4$};
%\path[draw] (2,2.67) -- (6,4)  node[pos=0.5, sloped,above] {Lem.3.5};
%\path[draw] (2, 0) -- (10,4)   node[pos=0.5, sloped,above] {Prop.5.1};
%\path[draw] (6,4) -- (10,8) node[pos=0.1,left] {$E_L$} node[pos=0.5, sloped,above] {$T \simeq T'$};
\path[draw,line width=1.2pt] (2,2) -- (4.67,2.67); 
\path[draw,line width=1.2pt] (4.67,2.67) -- (6,4); 
\path[draw,line width=1.2pt] (7.33,4) -- (10,5); 
\path[draw,line width=1.2pt] (6,4) -- (7.33,4);
\path[draw,line width=1.2pt] (10,1.6) -- (10,3.2);
%\path[draw,line width=1.2pt] (6,4) -- (10,5.33);
\path[fill=black] (2,0)--(10,0)--(10,1.6)--cycle;
\path[fill=black] (0,0)--(2,0)--(2,1.6)--(0,1.2)-- cycle;
%\path[shade,draw] (2,0)--(10,2)--(10,4)-- cycle;
\fill (4.67,2.67) circle (2pt);
\fill (10,8) circle (2pt);
\fill (6,4) circle (2pt);
\fill (10,4) circle (2pt);
\fill (10,4.8) circle (1pt);
\fill (2,2) circle (2pt);
\path[draw][dashed] (6,0) -- (6,8);
\path[draw][dashed] (2,4) -- (10,4) node[pos=0.5,below]{$U$} node[pos=0,left]{$3/2$};
\path[draw][dashed] (4.67,0) -- (4.67,8) node[pos=0.33,below]{$E_L$} node[pos=0,below]{$7/3$};
\path[draw][dashed] (7.33,0) -- (7.33,8) node[pos=0,below]{$8/3$};
\path[draw][dashed] (2,2.67) -- (10,2.67);
\path[draw][dashed] (2,1.6) -- (10,1.6);
\path[draw][dashed] (2,2) -- (10,2) node[pos=0,left]{$D(K_C)$} node[pos=1,right]{$5/4$};

%Prop4.13(i)
\fill (8,2) circle (1pt);
\fill (8.4,1.6) circle (1pt);
\fill (8.67,1.33) circle (1pt);

\fill (9,3) circle (1pt);
\fill (9.11,2.67) circle (1pt);
\fill (9.2,2.4) circle (1pt);
\fill (9.27,2.18) circle (1pt);
\fill (9.33,2) circle (1pt);
\fill (9.38,1.85) circle (1pt);
\fill (9.43,1.71) circle (1pt);
\fill (9.47,1.6) circle (1pt);
\fill (9.5,1.5) circle (1pt);
\fill (9.53,1.41) circle (1pt);
\fill (9.56,1.33) circle (1pt);

\fill (9.33,3.33) circle (1pt);
\fill (9.38,3.08) circle (1pt);
\fill (9.43,2.86) circle (1pt);
\fill (9.47,2.67) circle (1pt);
\fill (9.5,2.5) circle (1pt);
\fill (9.53,2.35) circle (1pt);
\fill (9.56,2.22) circle (1pt);
\fill (9.58,2.11) circle (1pt);
\fill (9.6,2) circle (1pt);
\fill (9.62,1.9) circle (1pt);
\fill (9.64,1.82) circle (1pt);
\fill (9.65,1.74) circle (1pt);
\fill (9.67,1.67) circle (1pt);
\fill (9.69,1.54) circle (1pt);
\fill (9.70,1.48) circle (1pt);
\fill (9.71,1.43) circle (1pt);

\fill (9.5,3.5) circle (1pt);
\fill (9.53,3.29) circle (1pt);
\fill (9.56,3.11) circle (1pt);
\fill (9.58,2.95) circle (1pt);
\fill (9.6,2.8) circle (1pt);
\fill (9.62,2.67) circle (1pt);
\fill (9.64,2.55) circle (1pt);
\fill (9.65,2.43) circle (1pt);
\fill (9.67,2.33) circle (1pt);
\fill (9.68,2.24) circle (1pt);
\fill (9.69,2.15) circle (1pt);
\fill (9.7,2.07) circle (1pt);
\fill (9.71,2) circle (1pt);
\fill (9.72,1.93) circle (1pt);
\fill (9.73,1.87) circle (1pt);
\fill (9.74,1.8) circle (1pt);
\path[draw,line width=1.2pt] (9.74,1.8) -- (10,0);

%Prop4.13(ii)
\fill (8.22,1) circle (1pt);
\fill (8.4,1) circle (1pt);
\fill (8.55,1) circle (1pt);
\fill (8.67,1) circle (1pt);

\fill (9.05,1.88) circle (1pt);
\fill (9.11,1.78) circle (1pt);
\fill (9.15,1.68) circle (1pt);
\fill (9.2,1.6) circle (1pt);
\fill (9.23,1.52) circle (1pt);
\fill (9.27,1.45) circle (1pt);
\path[draw,line width=1.2pt] (9.27,1.45) -- (10,0);

\fill (8.77,1.23) circle (1pt);
\fill (8.86,1.14) circle (1pt);
\fill (8.93,1.07) circle (1pt);
\fill (9,1) circle (1pt);

%Prop4.10(iii)
\fill (8.4,1.6) circle (1pt);
\fill (9.2,2.4) circle (1pt);
\fill (9.47,2.67) circle (1pt);

%\prop4.13
\fill (8,2) circle (1pt);
\fill (8.67,2.67) circle (1pt);
\fill (9,3) circle (1pt);
\fill (9.2,3.2) circle (1pt);
\fill (9.33,3.33) circle (1pt);
\fill (9.43,3.43) circle (1pt);
\fill (9.5,3.5) circle (1pt);
\fill (9.56,3.56) circle (1pt);
\fill (9.6,3.6) circle (1pt);
\path[draw,line width=1.2pt] (9.64,3.64) -- (10,4);

%Prop4.12
\fill (8,4) circle (1pt);

%Prop4.14
\fill (7.33,2.67) circle (1pt);

%Prop3.8(i)
\fill (4,2) circle (1pt);
\fill (3.6,1.6) circle (1pt);
\fill (3.33,1.33) circle (1pt);
\fill (3,1) circle (1pt);
\fill (2.89,0.89) circle (1pt);
\fill (2.8,0.8) circle (1pt);
\fill (2.73,0.73) circle (1pt);
\fill (2.67,0.67) circle (1pt);
\fill (2.62,0.62) circle (1pt);
\fill (2.57,0.57) circle (1pt);
\fill (2.53,0.53) circle (1pt);
\fill (2.5,0.5) circle (1pt);
\fill (2.47,0.47) circle (1pt);
\path[draw,line width=1.2pt] (2,0) -- (2.47,0.47);

\fill (3,2) circle (1pt);
\fill (2.8,1.6) circle (1pt);
\fill (2.73,1.45) circle (1pt);
\fill (2.67,1.33) circle (1pt);
\fill (2.62,1.23) circle (1pt);
\fill (2.57,1.14) circle (1pt);
%\fill (2.53,1.97) circle (1pt);
\fill (2.5,1) circle (1pt);
\fill (2.47,0.94) circle (1pt);
\fill (2.44,0.88) circle (1pt);
\path[draw,line width=1.2pt] (2,0) -- (2.44,0.88);

\fill (2.67,2) circle (1pt);
\fill (2.62,1.85) circle (1pt);
\fill (2.57,1.71) circle (1pt);
\fill (2.53,1.6) circle (1pt);
\fill (2.5,1.5) circle (1pt);
\fill (2.47,1.41) circle (1pt);
\fill (2.44,1.33) circle (1pt);
\fill (2.42,1.26) circle (1pt);
\fill (2.4,1.2) circle (1pt);
\fill (2.38,1.14) circle (1pt);
\path[draw,line width=1.2pt] (2,0) -- (2.38,1.14);

\fill (2.5,2) circle (1pt);
\fill (2.47,1.88) circle (1pt);
\fill (2.44,1.78) circle (1pt);
\fill (2.42,1.68) circle (1pt);
\fill (2.4,1.6) circle (1pt);
\fill (2.38,1.52) circle (1pt);
\fill (2.36,1.45) circle (1pt);
\fill (2.35,1.39) circle (1pt);
\path[draw,line width=1.2pt] (2,0) -- (2.35,1.39);

%Prop3.8(ii)
\fill (3.78,0.89) circle (1pt);
\fill (3.6,0.8) circle (1pt);
\fill (3.45,0.73) circle (1pt);
\fill (3.33,0.67) circle (1pt);
\fill (3.23,0.62) circle (1pt);
\fill (3.14,0.57) circle (1pt);
\fill (3.07,0.53) circle (1pt);
\fill (3,0.5) circle (1pt);
\fill (2.94,0.47) circle (1pt);
\fill (2.89,0.44) circle (1pt);
\fill (2.84,0.42) circle (1pt);
\fill (2.8,0.4) circle (1pt);
\path[draw,line width=1.2pt] (2,0) -- (2.8,0.4);

\fill (3.23,1.23) circle (1pt);
\fill (3.14,1.14) circle (1pt);
\fill (3.07,1.07) circle (1pt);
\fill (3,1) circle (1pt);
\fill (2.94,0.94) circle (1pt);
\fill (2.89,0.89) circle (1pt);
\fill (2.84,0.84) circle (1pt);
\fill (2.8,0.8) circle (1pt);
\path[draw,line width=1.2pt] (2,0) -- (2.8,0.8);

\fill (2.94,1.41) circle (1pt);
\fill (2.89,1.33) circle (1pt);
\fill (2.84,1.26) circle (1pt);
\fill (2.8,1.2) circle (1pt);
\fill (2.76,1.14) circle (1pt);
\fill (2.73,1.09) circle (1pt);
\fill (2.7,1.04) circle (1pt);
\fill (2.67,1) circle (1pt);
\path[draw,line width=1.2pt] (2,0) -- (2.67,1);

\fill (2.76,1.52) circle (1pt);
\fill (2.73,1.45) circle (1pt);
\fill (2.45,1.39) circle (1pt);
\fill (2.67,1.33) circle (1pt);
\fill (2.64,1.28) circle (1pt);
\fill (2.62,1.23) circle (1pt);
\fill (2.59,1.19) circle (1pt);
\fill (2.57,1.14) circle (1pt);
\path[draw,line width=1.2pt] (2,0) -- (2.57,1.14);

%\prop3.8(iii), m=5
\fill (3.6,1.6) circle (1pt);
\fill (2.8,1.6) circle (1pt);
\fill (2.53,1.6) circle (1pt);
\fill (2.4,1.6) circle (1pt);
\fill (2.32,1.6) circle (1pt);
\fill (2.27,1.6) circle (1pt);
\fill (2.23,1.6) circle (1pt);
\path[draw,line width=1.2pt] (2,1.6) -- (2.23,1.6);

\path[draw] (2,2) .. controls (4.67,2.6) .. (10,4.5);

\path[draw,line width=0.2pt] (2.1,-1.3) -- (7.9,-1.3) node[pos=0.5,sloped,above] {Figure 1: $C$ trigonal, $2 \leq \mu \leq 3$};

\end{tikzpicture}

 \begin{tikzpicture}
\path[draw][->] (0,-0.05) -- (0,11)    node[pos=1,left] {$\lambda$}   node[pos=0,below] {$3$}  node[pos=0.725,left]{$2$}   node[pos=0.71,right]{$T$}
 node[pos=0.44,left]{$8/5$} ;
\path[draw][->] (-0.2,0 ) -- (8.2,0) node[pos=1.02,right] {$\mu$}  node[pos=0.98,below] {$4$} node[pos=0,left] {$1$}; 
               
\path[draw][dashed] (0,8) -- (8,8);
\path[draw] (8,0) -- (8,11.3)  node[pos=0.7,right]{$Q$} node[pos=0.99,right] {$12/5$};
\path[draw,line width=1.2pt] (0,8) -- (8,11.2);

\path[fill=black] (0,0)--(8,0)--(8,8)--cycle;
\path[shade,draw] (0,0)--(8,8)--(0,4.8)--cycle;

\fill (8,8) circle (2pt);
\fill (0,8) circle (2pt);

%\path[draw] (0,4.5) .. controls (4.67,2.6) .. (8,8);
\path[draw] (0,4.49) .. controls (4,6.124) .. (8,8);

\path[draw,line width=0.2pt] (1.65,-1.1) -- (7.35,-1.1) node[pos=0.5,sloped,above] {Figure 2: $C$ trigonal, $3 < \mu \leq 4$}; 

\end{tikzpicture}

The thicker solid lines indicate 
the upper bounds for non-emptiness, given by Theorem \ref{thm4.7}.

The shaded areas consist of 
points $(\mu, \lambda)$ for which there exist $(n,d,k)$ with 
$$
\frac{d}{n} = \mu,\quad  \frac{k}{n} =
 \lambda \quad \mbox{and} \quad B(n,d,k) \neq \emptyset.
 $$
 The black areas are given by Proposition \ref{pln2}, Corollary \ref{c2.2} and Proposition  \ref{p3.7} 
and all $B(n,d,k)$ corresponding to points in these areas are non-empty. 
Note that the vertical line at $\mu = 4$ in Figure 2 is not included. The vertical line at $\mu = 3$ in Figure 1 
ending at $\lambda = \frac{7}{5}$ corresponds to Proposition \ref{p4.4}, but not all $B(n,3n,k)$ corresponding to points on this line are non-empty. 

In the grey area, which 
corresponds to Proposition \ref{prop2.4}, there are some $(n,d,k)$ for which possibly 
$B(n,d,k) = \emptyset$. 
However, for any $(\mu,\lambda)$ in this area, there exist $(n,d,k)$ with $\mu = \frac{d}{n},\; \lambda = \frac{k}{n}$ such that $B(n,d,k) \neq \emptyset$.

The dots represent points for which some $B(n,d,k) \neq \emptyset$. The series of dots  
arise from Propositions \ref{p3.6}, \ref{p4.9}, \ref{p4.12}.
There are also isolated dots corresponding to Propositions \ref{p4.13}, \ref{p4.14} and \ref{p6.4}.
The dot at $(3,\frac{8}{5})$ may not represent a bundle, but is the upper bound established in Lemma \ref{l5.8}.

The BN-curve (the thin curve in the figures) given by $\lambda(\lambda - \mu  + 4) = 4$ (or $\beta(n,d,k) = 1$) passes through the points $
(2,-1 + \sqrt{5}), (\frac{7}{3},\frac{4}{3})$, $(\frac{8}{3}, \frac{1}{3}(-2 + \sqrt{40}))$, $(3,\frac12(-1+\sqrt{17}))$ and $(4,2)$. The bundle $D(K_C)$ in Figures 1 and 3 lies marginally above the curve and corresponds to the value $\beta=0$; the bundles $U$ and $T$ in Figure 1 correspond to the value $\beta=-1$.
All the bundles constructed in this paper in the case $\Cl(C) = 2$ have $\beta(n,d,k) \geq 0$, but this does not rule out the possibility that $B(n,d,k)$ could be non-empty for some 
$(n,d,k)$ with $\beta(n,d,k) < 0$ even in this case.
%\newpage

We turn now to the case of Clifford index $2$, represented by Figures 3 and 4. Note that in Figure 3, the thick upper line applies to any curve of Clifford index  2 (Theorem \ref{thm5.2}) and the lower line to a general curve (Theorem \ref{thm5.3}). The vertical line at $\mu = 4$ in Figure 4 is not included. The series of dots arise from  Corollaries \ref{c5.5}, \ref{c5.6} and Proposition \ref{p5.7}.

\begin{tikzpicture}
\path[draw][->] (2,-0.2) -- (2,8.2) node[pos=0.98,left] {$\lambda$} node[pos=.34,left] {$4/3$};
\path[draw][->] (0,0 ) -- (10.2,0) node[pos=1.02,right] {$\mu$} node[pos=.6,below] {$5/2$} node[pos=0.21,below] {$2$} node[pos=0.98,below] {$3$} node[pos=0,left] {$1$}; 
               
\path[draw] (10, 0) -- (10,8) node[pos=0.67,right] {$5/3$} node[pos=0.62,right] {$13/8$} node[pos=0.2,right] {$6/5$} 
node[pos=0.5,right] {\eqref{e5.1}} node[pos=1,right] {$2$} node[pos=0.25,right] {$D(K_C)(p)$}; 
%\path[draw] (2, 0) -- (10,2);
\path[draw][dashed] (2,8) -- (10,8);
%\path[draw][dashed] (2,6) -- (10,6)  node[pos=0,left] {$7/4$};
%\path[draw] (2,2.67) -- (6,4)  node[pos=0.5, sloped,above] {Lem.3.5};
%\path[draw] (2, 0) -- (10,4)   node[pos=0.5, sloped,above] {Prop.5.1};
%\path[draw] (6,4) -- (10,8) node[pos=0.1,left] {$E_L$} node[pos=0.5, sloped,above] {$T \simeq T'$};
\path[draw,line width=1.2pt] (2,2) -- (4.67,2.67); 
\path[draw,line width=1.2pt] (4.67,2.67) -- (6,4); 
\path[draw,line width=1.2pt] (7.33,4) -- (10,5); 
\path[draw,line width=1.2pt] (6,4) -- (7.33,4);
%\path[draw,line width=1.2pt] (10,1.6) -- (10,3.2);
\path[draw,line width=1.2pt] (6,4) -- (10,5.33);
\path[fill=black] (2,0)--(10,0)--(10,1.6)--cycle;
\path[fill=black] (0,0)--(2,0)--(2,1.6)--(0,1.2)-- cycle;
%\path[shade,draw] (2,0)--(10,2)--(10,4)-- cycle;
\fill (4.67,2.67) circle (2pt);
\fill (10,2) circle (2pt);
\fill (6,2) circle (1.2pt);
\fill (10,2.67) circle (1.2pt);
\fill (10,4) circle (2pt);
%\fill (10,5.1) circle (1pt);
\fill (2,2) circle (2pt);
\path[draw][dashed] (6,0) -- (6,8);
\path[draw][dashed] (2,4) -- (10,4) node[pos=0,left]{$3/2$};
\path[draw][dashed] (4.67,0) -- (4.67,8) node[pos=0.33,below]{$E_L$} node[pos=0,below]{$7/3$};
\path[draw][dashed] (7.33,0) -- (7.33,8) node[pos=0,below]{$8/3$};
\path[draw][dashed] (2,2.67) -- (10,2.67);
\path[draw][dashed] (2,1.6) -- (10,1.6);
\path[draw][dashed] (2,2) -- (10,2) node[pos=0,left]{$D(K_C)$};

%Prop3.8(i)
\fill (4,2) circle (1pt);
\fill (3.6,1.6) circle (1pt);
\fill (3.33,1.33) circle (1pt);
\fill (3,1) circle (1pt);
\fill (2.89,0.89) circle (1pt);
\fill (2.8,0.8) circle (1pt);
\fill (2.73,0.73) circle (1pt);
\fill (2.67,0.67) circle (1pt);
\fill (2.62,0.62) circle (1pt);
\fill (2.57,0.57) circle (1pt);
\fill (2.53,0.53) circle (1pt);
\fill (2.5,0.5) circle (1pt);
\fill (2.47,0.47) circle (1pt);
\path[draw,line width=1.2pt] (2,0) -- (2.47,0.47);

\fill (3,2) circle (1pt);
\fill (2.8,1.6) circle (1pt);
\fill (2.73,1.45) circle (1pt);
\fill (2.67,1.33) circle (1pt);
\fill (2.62,1.23) circle (1pt);
\fill (2.57,1.14) circle (1pt);
\fill (2.53,1.97) circle (1pt);
\fill (2.5,1) circle (1pt);
\fill (2.47,0.94) circle (1pt);
\fill (2.44,0.88) circle (1pt);
\path[draw,line width=1.2pt] (2,0) -- (2.44,0.88);

\fill (2.67,2) circle (1pt);
\fill (2.62,1.85) circle (1pt);
\fill (2.57,1.71) circle (1pt);
\fill (2.53,1.6) circle (1pt);
\fill (2.5,1.5) circle (1pt);
\fill (2.47,1.41) circle (1pt);
\fill (2.44,1.33) circle (1pt);
\fill (2.42,1.26) circle (1pt);
\fill (2.4,1.2) circle (1pt);
\fill (2.38,1.14) circle (1pt);
\path[draw,line width=1.2pt] (2,0) -- (2.38,1.14);

\fill (2.5,2) circle (1pt);
\fill (2.47,1.88) circle (1pt);
\fill (2.44,1.78) circle (1pt);
\fill (2.42,1.68) circle (1pt);
\fill (2.4,1.6) circle (1pt);
\fill (2.38,1.52) circle (1pt);
\fill (2.36,1.45) circle (1pt);
\fill (2.35,1.39) circle (1pt);
\path[draw,line width=1.2pt] (2,0) -- (2.35,1.39);

%Prop3.8(ii)
\fill (3.78,0.89) circle (1pt);
\fill (3.6,0.8) circle (1pt);
\fill (3.45,0.73) circle (1pt);
\fill (3.33,0.67) circle (1pt);
\fill (3.23,0.62) circle (1pt);
\fill (3.14,0.57) circle (1pt);
\fill (3.07,0.53) circle (1pt);
\fill (3,0.5) circle (1pt);
\fill (2.94,0.47) circle (1pt);
\fill (2.89,0.44) circle (1pt);
\fill (2.84,0.42) circle (1pt);
\fill (2.8,0.4) circle (1pt);
\path[draw,line width=1.2pt] (2,0) -- (2.8,0.4);

\fill (3.23,1.23) circle (1pt);
\fill (3.14,1.14) circle (1pt);
\fill (3.07,1.07) circle (1pt);
\fill (3,1) circle (1pt);
\fill (2.94,0.94) circle (1pt);
\fill (2.89,0.89) circle (1pt);
\fill (2.84,0.84) circle (1pt);
\fill (2.8,0.8) circle (1pt);
\path[draw,line width=1.2pt] (2,0) -- (2.8,0.8);

\fill (2.94,1.41) circle (1pt);
\fill (2.89,1.33) circle (1pt);
\fill (2.84,1.26) circle (1pt);
\fill (2.8,1.2) circle (1pt);
\fill (2.76,1.14) circle (1pt);
\fill (2.73,1.09) circle (1pt);
\fill (2.7,1.04) circle (1pt);
\fill (2.67,1) circle (1pt);
\path[draw,line width=1.2pt] (2,0) -- (2.67,1);

\fill (2.76,1.52) circle (1pt);
\fill (2.73,1.45) circle (1pt);
\fill (2.45,1.39) circle (1pt);
\fill (2.67,1.33) circle (1pt);
\fill (2.64,1.28) circle (1pt);
\fill (2.62,1.23) circle (1pt);
\fill (2.59,1.19) circle (1pt);
\fill (2.57,1.14) circle (1pt);
\path[draw,line width=1.2pt] (2,0) -- (2.57,1.14);

%\prop3.8(iii), m=5
\fill (3.6,1.6) circle (1pt);
\fill (2.8,1.6) circle (1pt);
\fill (2.53,1.6) circle (1pt);
\fill (2.4,1.6) circle (1pt);
\fill (2.32,1.6) circle (1pt);
\fill (2.27,1.6) circle (1pt);
\fill (2.23,1.6) circle (1pt);
\path[draw,line width=1.2pt] (2,1.6) -- (2.23,1.6);

%Corollary 5.6
\fill (8,2) circle (1pt);
\fill (9,3) circle (1pt);
\fill (9.2,3.2) circle (1pt);
\fill (9.33,3.33) circle (1pt);
\fill (9.43,3.43) circle (1pt);
\fill (9.5,3.5) circle (1pt);
\fill (9.56,3.56) circle (1pt);
\fill (9.6,3.6) circle (1pt);
\fill (9.64,3.64) circle (1pt);
\fill (9.67,3.67) circle (1pt);
\path[draw,line width=1.2pt] (9.67,3.67) -- (10,4);

%prop5.9
\fill (7.33,2.67) circle (1pt);
\fill (6,2.67) circle (1pt);
\fill (5.56,2.67) circle (1pt);
\fill (5.33,2.67) circle (1pt);
\fill (5.22,2.67) circle (1pt);
\fill (5.11,2.67) circle (1pt);
\fill (5.05,2.67) circle (1pt);
\fill (4.96,2.67) circle (1pt);
\fill (5,2.67) circle (1pt);
\fill (4.93,2.67) circle (1pt);
\fill (4.91,2.67) circle (1pt);
\path[draw,line width=1.2pt] (4.67,2.67) -- (4.91,2.67);

\path[draw] (2,2) .. controls (4.67,2.6) .. (10,4.5);

\path[draw,line width=0.2pt] (2.4,-1.3) -- (8.5,-1.3) node[pos=0.5,sloped,above] {Figure 3: $\Cl(C) = 2,\; 2 \leq \mu \leq 3$}; 

\end{tikzpicture}

\bigskip

\bigskip

 \begin{tikzpicture}
\path[draw][->] (0,-0.05) -- (0,8)    node[pos=1,above] {$\lambda$}   node[pos=0,below] {$3$}  node[pos=1,left]{$2$}  node[pos=0.67,left] {$5/3$};
\path[draw][->] (-0.2,0 ) -- (8.2,0) node[pos=1.02,right] {$\mu$}  node[pos=0.98,below] {$4$} node[pos=0,left] {$1$}; 
               
\path[draw][dashed] (0,8) -- (8,8);
\path[draw] (8,0) -- (8,8)  node[pos=1,right] {$Q$};
\path[draw,line width=1.2pt] (0,5.33) -- (8,8);

\path[fill=black] (0,0)--(8,0)--(8,8)--cycle;
%\path[shade,draw] (0,0)--(8,8)--(0,4.8)--cycle;

\fill (8,8) circle (2pt);
%\fill (0,8) circle (2pt);

%Prop5.6
\fill (4,4) circle (1pt);
\fill (2,4) circle (1pt);
\fill (1.33,4) circle (1pt);
\fill (1,4) circle (1pt);
\fill (0.8,4) circle (1pt);
\fill (0.67,4) circle (1pt);
\fill (0.57,4) circle (1pt);
\fill (0.5,4) circle (1pt);
\fill (0.44,4) circle (1pt);
\fill (0.4,4) circle (1pt);
\fill (0.36,4) circle (1pt);
\path[draw,line width=1.3pt] (0,4) -- (0.36,4);

\path[draw] (0,4.49) .. controls (4,6.124) .. (8,8);

\path[draw,line width=0.2pt] (0.45,-1.1) -- (6.55,-1.1) node[pos=0.5,sloped,above] {Figure 4: $\Cl(C) = 2,\; 3< \mu \leq 4$}; 

\end{tikzpicture}

\end{document}